\documentclass[12pt]{amsart}
\usepackage{amsfonts, amsmath, fullpage, url, stmaryrd}

\newtheorem{theorem}{Theorem}[section]
\newtheorem{lemma}[theorem]{Lemma}
\newtheorem{cor}[theorem]{Corollary}
\newtheorem{conj}[theorem]{Conjecture}

\theoremstyle{definition}
\newtheorem{defn}[theorem]{Definition}
\newtheorem{hypothesis}[theorem]{Hypothesis}
\newtheorem{remark}[theorem]{Remark}
\newtheorem{example}[theorem]{Example}
\newtheorem{convention}[theorem]{Convention}

\numberwithin{equation}{theorem}

\newcommand{\be}{\mathbf{e}}
\newcommand{\bv}{\mathbf{v}}

\newcommand{\dual}{\vee}

\newcommand{\calE}{\mathcal{E}}

\newcommand{\calH}{\mathcal{H}}
\newcommand{\calN}{\mathcal{N}}
\newcommand{\calNP}{\mathcal{NP}}
\newcommand{\calO}{\mathcal{O}}
\newcommand{\calP}{\mathcal{P}}
\newcommand{\calR}{\mathcal{R}}
\newcommand{\calS}{\mathcal{S}}
\newcommand{\CC}{\mathbb{C}}
\newcommand{\DD}{\mathbb{D}}
\newcommand{\FF}{\mathbb{F}}
\newcommand{\PP}{\mathbb{P}}
\newcommand{\QQ}{\mathbb{Q}}
\newcommand{\RR}{\mathbb{R}}
\newcommand{\ZZ}{\mathbb{Z}}

\newcommand{\frako}{\mathfrak{o}}

\DeclareMathOperator{\an}{an}
\DeclareMathOperator{\bd}{bd}
\DeclareMathOperator{\coker}{coker}
\DeclareMathOperator{\dR}{dR}

\DeclareMathOperator{\Frac}{Frac}
\DeclareMathOperator{\Gal}{Gal}
\DeclareMathOperator{\GL}{GL}

\DeclareMathOperator{\inte}{int}
\DeclareMathOperator{\Irr}{Irr}
\DeclareMathOperator{\length}{length}

\DeclareMathOperator{\sep}{sep}

\DeclareMathOperator{\val}{val}

\begin{document}

\title{Convergence polygons for connections on nonarchimedean curves}
\author{Kiran S. Kedlaya}
\date{July 10, 2015}
\thanks{Thanks to Francesco Baldassarri, Andrew Obus, Andrea Pulita, Junecue Suh, Daniele Turchetti, the anonymous referee, and the participants of the Simons Symposium for detailed feedback. The author was supported by NSF grant DMS-1101343 and UC San Diego
(Warschawski Professorship), and by a Simons Visiting Professorship grant during November 2014 (visiting Pulita).}

\maketitle

In classical analysis, one builds the catalog of special functions by repeatedly adjoining solutions of differential equations whose coefficients are previously known functions. Consequently, the properties of special functions depend crucially on the basic properties of ordinary differential equations.
This naturally led to the study of formal differential equations, as in the seminal work of Turrittin \cite{turrittin}; this may be viewed retroactively as a theory of differential equations over a trivially valued field.
After the introduction of $p$-adic analysis in the early 20th century, there began to be corresponding interest in solutions of $p$-adic differential equations; however, aside from some isolated instances (e.g., the proof of the Nagell-Lutz theorem; see Theorem~\ref{T:p-adic Cauchy}), a unified theory of $p$-adic ordinary differential equations did not emerge until the pioneering work of Dwork on the relationship between $p$-adic special functions and the zeta functions of algebraic varieties over finite fields (e.g., see \cite{dwork-zeta2, dwork-cycles}). 
At that point, serious attention began to be devoted to a 
serious discrepancy between the $p$-adic and complex-analytic theories:  on an open $p$-adic disc, a nonsingular differential equation can have a formal solution which does not converge in the entire disc (e.g., the exponential series).
One is thus led to quantify the convergence of power series solutions of differential equations involving rational functions over a nonarchimedean field;
this was originally done by Dwork in terms of the \emph{generic radius of convergence}
\cite{dwork-pde2}. This and more refined invariants were studied 
by numerous authors during the half-century following Dwork's initial work, as documented in the author's book \cite{kedlaya-book}.

At around the time that \cite{kedlaya-book} was published, a new perspective was introduced by Baldassarri \cite{baldassarri} 
(and partly anticipated in prior unpublished work of Baldassarri and Di Vizio \cite{baldassarri-divizio})
which makes full use of Berkovich's theory of nonarchimedean analytic spaces. Given a differential equation as above, or more generally a connection on a curve over a nonarchimedean field, one can define an invariant called the \emph{convergence polygon}; this is a function from the underlying Berkovich topological space of the curve into a space of Newton polygons, which measures the convergence of formal horizontal sections and is well-behaved with respect to both the topology and the piecewise linear structure on the Berkovich space. 
One can translate much of the prior theory of $p$-adic differential equations into (deceptively) simple statements about the behavior of the convergence polygon;
this process was carried out in a series of papers by Poineau and Pulita
\cite{pulita-poineau, pulita-poineau2, pulita-poineau4, pulita-poineau5},
as supplemented by work of this author \cite{kedlaya-radii}
and
upcoming joint work with Baldassarri \cite{baldassarri-kedlaya}.

In this paper, we present the basic theorems on the convergence polygon, 
which provide a number of combinatorial constraints that may be used to extract information about convergence of formal horizontal sections at one point from corresponding information at other points. We include numerous examples to illustrate some typical behaviors of the convergence polygon. We also indicate some relationships between convergence polygons, index formulas for de Rham cohomology, and the geometry of finite morphisms, paying special attention to the case of cyclic $p$-power coverings with $p$ equal to the residual characteristic.
This case is closely linked with the Oort lifting problem for Galois covers of curves in characteristic $p$, and some combinatorial constructions arising in that problem turn out to be closely related to convergence polygons. There are additional applications
to the study of integrable connections on higher-dimensional nonarchimedean analytic spaces, both in the cases of zero residual characteristic \cite{kedlaya-goodformal2} and positive residual characteristic \cite{kedlaya-semistable4}, but we do not pursue these applications here.

To streamline the exposition, we make no attempt to indicate the techniques of proof underlying our main results; in some cases, quite sophisticated arguments are required.
We limit ourselves to saying that the basic tools are developed in a self-contained fashion in \cite{kedlaya-book} (but without reference to Berkovich spaces), and the other aforementioned results are obtained by combining results from \cite{kedlaya-book} in an intricate manner.
(Some exceptions are made for results which do not occur in any existing paper; their proofs are relegated to an appendix.)
We also restrict generality by considering only proper curves, even though many of the results we discuss can be formulated for open curves, possibly of infinite genus.

At the suggestion of the referee, we include (as Appendix~\ref{sec:thematic bib}) a \emph{thematic bibliography}
in the style of \cite{robba-cohom}
listing additional references germane to the topics discussed in the article.
We also include references for the following topics which we have omitted out of space considerations, even though in practice they are thoroughly entangled with the proofs of the results which we do mention.
\begin{itemize}
\item
\emph{Decomposition theorems}, which give splittings of connections analogous to the factorizations of polynomials given by various forms of Hensel's lemma.
\item
\emph{Monodromy theorems}, which provide structure theorems for connections at the expense of passing from a given space to a finite cover.
\item
\emph{Logarithmic growth}, i.e., secondary terms in the measurement of convergence of horizontal sections.
\end{itemize}

\section{Newton polygons}

As setup for our definition of convergence polygons, we fix some conventions regarding Newton polygons. 
\begin{defn}
For $n$ a positive integer, let $\calP[0,n]$ be the set of continuous functions $\calN: [0,n] \to \RR$ satisfying the following conditions.
\begin{enumerate}
\item[(a)]
We have $\calN(0) = 0$.
\item[(b)]
For $i=1,\dots,n$, the restriction of $\calN$ to $[i-1,i]$ is affine. 
\end{enumerate}
For $i=1,\dots,n$, we write $h_i: \calP[0,n] \to \RR$ for the function $\calN \mapsto \calN(i)$;
we call $h_i(\calN)$ the $i$-th \emph{height} of $\calN$. The product map
$h_1 \times \cdots \times h_n: \calP[0,n] \to \RR^n$ is a bijection, using which we equip $\calP[0,n]$ with a topology and an integral piecewise linear structure. 
We sometimes refer to $h_n$ simply as $h$ and call it the \emph{total height}.
\end{defn}

\begin{defn} \label{D:slope}
Let $\calNP[0,n]$ be the subset of $\calP[0,n]$ consisting of \emph{concave} functions.
For $i=1,\dots,n$,
we write $s_i: \calP[0,n] \to \RR$ for the function $\calN \mapsto \calN(i) - \calN(i-1)$;
we call $s_i(\calN)$ the $i$-th \emph{slope} of $\calN$. For $\calN \in \calNP[0,n]$, we have $s_1(\calN) \geq \cdots \geq s_n(\calN)$.
\end{defn}

\begin{defn}
Let $I \subseteq \RR$ be a closed interval. A function $\calN: I \to \calNP[0,n]$ is \emph{affine} if it has the form $\calN(t) = \calN_0 + t \calN_1$ for some $\calN_0, \calN_1 \in  \calP[0,n]$. In this case, we call $\calN_1$ the \emph{slope} of $\calN$.
We say that $\calN$ has \emph{integral derivative} if $\calN_1(i) \in \ZZ$ for $i=n$ and for each $i \in \{1,\dots,n-1\}$ such that for all $t$ in the interior of $I$, $\calN(t)$ has a change of slope at $i$. 
This implies that the graph of $\calN_1$ has vertices only at lattice points, but not conversely.
(It would be natural to use the terminology \emph{integral slope}, but we avoid this terminology to alleviate confusion with Definition~\ref{D:slope}.)
\end{defn}

\section{PL structures on Berkovich curves}

We next recall the canonical piecewise linear structure on a Berkovich curve
(e.g., see \cite{bpr}).

\begin{hypothesis}
For the rest of this paper, let $K$ be a \emph{nonarchimedean field},
i.e., a field complete with respect to a nonarchimedean absolute value;
let $X$ be a smooth, proper, geometrically connected curve over $K$;
let $Z$ be a finite set of closed points of $X$;
and put $U = X - Z$.
\end{hypothesis}

\begin{convention}
Whenever we view $\QQ_p$ as a nonarchimedean field, we normalize the $p$-adic absolute value so that $|p| = p^{-1}$.
\end{convention}

\begin{remark}
Recall that the points of the Berkovich analytification $X^{\an}$ may be identified with equivalence classes of pairs $(L,x)$ in which $L$ is a nonarchimedean field over $K$ and $x$ is an element of $X(L)$, where the equivalence relation is generated by relations of the form $(L,x) \sim (L', x')$ where $x'$ is the restriction of $x$ along a continuous $K$-algebra homomorphism $L \to L'$. As is customary, we classify points of $X^{\an}$ into types 1,2,3,4 (e.g., see \cite[Proposition~4.2.7]{kedlaya-radii}). To lighten notation, we identify $Z$ with $Z^{\an}$, which is a finite subset of $X^{\an}$ consisting of type 1 points.
\end{remark}

\begin{defn}
For $\rho >0$, let $x_\rho$ denote the generic point of the disc $|z| \leq \rho$ in $\PP^1_K$.
A \emph{segment} in $X^{\an}$ is a closed subspace $S$ homeomorphic to a closed interval for which there exist an open subspace $V$ of $X^{\an}$, a choice of values $0 \leq \alpha < \beta \leq +\infty$, and an isomorphism of $V$ with $\{z \in \PP^{1,\an}_K: \alpha < |z| < \beta\}$ identifying  the interior of $S$
with $\{x_\rho: \rho \in (\alpha,\beta)\}$. A \emph{virtual segment} in $X^{\an}$ is a connected closed subspace whose base extension to some finite extension of $K$ is a disjoint union of segments.

A \emph{strict skeleton} in $X^{\an}$ is a subspace $\Gamma$ containing $Z^{\an}$ equipped with a homeomorphism to a finite connected graph, such that each vertex of the graph corresponds to either a point of $Z$ or a point of type 2, and each edge corresponds to a virtual segment, and $X^{\an}$ retracts continuously onto $\Gamma$.
Using either tropicalizations or semistable models, one may realize
$X^{\an}$ as the inverse limit of its strict skeleta;
again, see \cite{bpr} for a detailed discussion.
\end{defn}

\begin{defn} \label{D:minimal strict skeleton}
Note that 
\[
\chi(U) = 2 - 2g(X) - \length(Z),
\]
so $\chi(U) \leq 0$ if and only if either $g(X) \geq 1$ or $\length(Z) \geq 2$.
In this case, there is a unique minimal strict skeleton in $X^{\an}$, which we denote
$\Gamma_{X,Z}$. Explicitly, if $K$ is algebraically closed, then the underlying set of $\Gamma_{X,Z}$ is the complement in $X^{\an}$ of the union of all open discs in $U^{\an}$;
for general $K$, the underlying set of $\Gamma_{X,Z}$ is the image under restriction of the minimal strict skelelon in $X_L^{\an}$ for $L$ a completed algebraic closure of $K$.
In particular, if $K'$ is the completion of an algebraic extension of $K$, then the minimal strict skeleton in $X_{K'}^{\an}$ is the inverse image of $\Gamma_{X,Z}$ in $X_{K'}^{\an}$.
However, the corresponding statement for a general nonarchimedean field extension $K'$ of $K$ is false;
see Definition~\ref{D:branch locus} for a related phenomenon.
\end{defn}

\section{Convergence polygons: projective line}
\label{sec:convergence p1}

We next introduce the concept of the convergence polygon associated to a differential equation on $\PP^1$. 
\begin{hypothesis} \label{H:convergence p1}
For the rest of this paper, we assume that the nonarchimedean field $K$ is of characteristic $0$, as otherwise the study of differential operators on $K$-algebras has a markedly different flavor (for instance, any derivation on a ring $R$ of characteristic $p>0$ has the subring of $p$-th powers in its kernel). By contrast, the residue characteristic of $K$, which we call $p$, may be either 0 or positive unless otherwise specified (e.g., if we refer to $\QQ_p$ then we implicitly require $p>0$).
\end{hypothesis}

\begin{hypothesis}
For the rest of \S\ref{sec:convergence p1}, 
take $X = \PP^1_K$, assume $\infty \in Z$,
and  consider the differential equation
\begin{equation} \label{eq:de}
y^{(n)} + f_{n-1}(z) y^{(n-1)} + \cdots + f_0(z) y = 0
\end{equation}
for some rational functions $f_0, \dots, f_{n-1} \in K(z)$ with poles only within $Z$.
If $Z = \{\infty\}$, let $m$ be the dimension of the $K$-vector space of entire solutions of \eqref{eq:de}; otherwise, take $m=0$.
\end{hypothesis}

\begin{defn}
For any nonarchimedean field $L$ over $K$ and any $x \in U(L)$, let
$S_x$ be the set of formal solutions of \eqref{eq:de} with $y \in L \llbracket z-x \rrbracket$.
By interpreting \eqref{eq:de} as a linear recurrence relation of order $n$ on the coefficients of a power series, we see that every list of $n$ initial conditions at $z=x$ corresponds to a unique formal solution; that is, the composition 
\[
S_x \to L \llbracket z-x \rrbracket \to L \llbracket z-x \rrbracket/(z-x)^n
\]
is a bijection. In particular, $S_x$ is an $L$-vector space of dimension $n$.
\end{defn}

\begin{theorem}[$p$-adic Cauchy theorem] \label{T:p-adic Cauchy}
Each element of $S_x$ has a positive radius of convergence.
\end{theorem}
\begin{proof}
This result was originally proved by Lutz \cite[Th\'eor\`eme~IV]{lutz} somewhat before the emergence of the general theory of $p$-adic differential equations; Lutz used it as a lemma in her proof of the \emph{Nagell-Lutz theorem} on the integrality of torsion points on rational elliptic curves. One can give several independent proofs using the modern theory;
see \cite[Proposition~9.3.3, Proposition~18.1.1]{kedlaya-book}.
\end{proof}

\begin{defn}
For $i=1,\dots,n-m$, choose $s_i(x) \in \RR$ so that $e^{-s_i(x)}$ is the supremum of the set of $\rho>0$ such that $U^{\an}$ contains the open disc $|z-x| < \rho$ and
$S_x$ contains $n-i+1$ linearly independent elements convergent on this disc. Note that this set is nonempty by Theorem~\ref{T:p-adic Cauchy} and bounded above by the definition of $m$, so the definition makes sense. 
In particular, $s_1(x)$ is the joint radius of convergence of all of the elements of $S_x$, while $s_{n-m}(x)$ is the maximum finite radius of convergence of a nonzero element of $S_x$.

Since $s_1(x) \geq \cdots \geq s_{n-m}(x)$, the $s_i(x)$ are the slopes of a polygon $\calN_z(x) \in \calNP[0,n-m]$,
which we call the \emph{convergence polygon} of \eqref{eq:de} at $x$. 
(We include $z$ in the notation to remind ourselves that $\calN_z$ depends on the choice of the coordinate $z$ of $X$.)
This construction is compatible with base change: if $L'$ is a nonarchimedean field containing $L$ and $x'$ is the image of $x$ in $U(L')$, then $\calN_z(x) = \calN_z(x')$. Consequently, we obtain a well-defined function $\calN_z: U^{\an} \to \calNP[0,n-m]$.
\end{defn}

\begin{defn}
Suppose that $Z \neq \{\infty\}$.
By definition, $e^{-s_1(\calN_z(x))}$ can never exceed the largest value of $\rho$ for which 
the disc $|z-x| < \rho$ does not meet $Z$. When equality occurs, we say that \eqref{eq:de} satisfies the \emph{Robba condition} at $x$.
\end{defn}

\begin{theorem} \label{T:continuous1}
The function $\calN_z: U^{\an} \to \calNP[0,n-m]$ is continuous; more precisely, it factors through the retraction of $\PP^{1,\an}_K$ onto some strict skeleton 
$\Gamma$, and the restriction of $\calN_z$ to each edge of $\Gamma$ is affine with integral derivative.
\end{theorem}
\begin{proof}
See \cite{pulita-poineau} or \cite{kedlaya-radii} or \cite{baldassarri-kedlaya}.
\end{proof}

One can say quite a bit more, but for this it is easier to shift to a coordinate-free interpretation, which also works for more general curves; see \S\ref{sec:convergence general}.

\section{A gallery of examples}
\label{sec:gallery p1}

To help the reader develop some intuition, we collect a few illustrative examples of convergence polygons. Throughout \S\ref{sec:gallery p1}, retain
Hypothesis~\ref{H:convergence p1}.

\begin{example} \label{exa:exponential}
Take $K = \QQ_p$, $Z = \{\infty\}$, and consider the differential equation $y'- y =0$. The formal solutions of this equation with $y \in L\llbracket z-x \rrbracket$ are the scalar multiples of the exponential series
\[
\exp(z-x) = \sum_{i=0}^\infty \frac{(z-x)^{i}}{i!},
\]
which has radius of convergence $p^{-1/(p-1)}$. Consequently,
\[
s_1(\calN_z(x)) = \frac{1}{p-1} \log p;
\]
in particular, $\calN_z$ is constant on $U^{\an}$.
\end{example}

In this next example, we illustrate the effect of changing $Z$ on the convergence polygon.
\begin{example} \label{exa:add pole}
Set notation as in Example~\ref{exa:exponential}, except now with $Z = \{0, \infty\}$. In this case we have
\[
s_1(\calN_z(x)) = \max\left\{-\log |x|, \frac{1}{(p-1)} \log p\right\}.
\]
In particular, $\calN_z$ 
factors through the retraction of $\PP^{1,\an}_K$ onto the path from 0 to $\infty$.
For $x \in U^{\an}$, the Robba condition holds at $x$ if and only if 
$|x| \leq p^{-1/(p-1)}$.
\end{example}

\begin{example} \label{exa:p-th root}
Take $K = \QQ_p$, $Z = \{0, \infty\}$, and consider the differential equation $y'- \frac{1}{p} z^{-1} y = 0$. The formal solutions of this equation with $y \in L\llbracket z-x \rrbracket$ are the scalar multiples of the binomial series
\[
\sum_{i=0}^\infty \binom{1/p}{i} x^{-i+p^{-1}}(z-x)^{i},
\]
which has radius of convergence $p^{-p/(p-1)} |x|$. Consequently,
\[
s_1(\calN_z(x)) = \frac{p}{p-1} \log p - \log |x|,
\]
so again $\calN_z$
factors through the retraction of $\PP^{1,\an}_K$ onto the path from 0 to $\infty$.
In this case, the Robba condition holds nowhere.
\end{example}

\begin{example}
Assume $p > 2$, take $K = \QQ_p$, $Z = \{0, \infty\}$, and consider the Bessel differential equation (with parameter $0$)
\[
y'' + z^{-1} y' + y = 0.
\]
This example was studied by Dwork \cite{dwork-bessel}, who showed that
\[
s_1(\calN_z(x)) = s_2(\calN_z(x)) = 
\max\left\{-\log |x|, \frac{1}{p-1} \log p \right\}.
\]
Again, $\calN_z$ 
factors through the retraction of $\PP^{1,\an}_K$ onto the path from 0 to $\infty$.
As in Example~\ref{exa:add pole},
for $x \in U^{\an}$, the Robba condition holds at $x$ if and only if 
$|x| \leq p^{-1/(p-1)}$.
\end{example}

Our next example illustrates a typical effect of varying a parameter.
\begin{example} \label{exa:logarithmic parameter}
Let $K$ be an extension of $\QQ_p$, take $Z = \{0, \infty\}$, and consider the differential equation $y' - \lambda z^{-1} y = 0$ for some $\lambda \in K$
(the case $\lambda = 1/p$ being Example~\ref{exa:p-th root}).
Then
\[
s_1(\calN_z(x)) = c + \frac{1}{p-1} \log p - \log |x|
\]
where $c$ is a continuous function of 
\[
c_0 = \min\{|\lambda - t|: t \in \ZZ_p\};
\]
namely, by \cite[Proposition~IV.7.3]{dgs} we have
\[
c = \begin{cases}
\log c_0 & \mbox{if $c_0 \geq 1$} \\
- \frac{p^m-1}{(p-1)p^m} \log p + 
\frac{1}{p^{m+1}} \log (p^m c_0) & \mbox{if $p^{-m-1} \leq c_0 \leq p^{-m}$} \qquad (m=0,1,\dots) \\
-\frac{1}{p-1} \log p & \mbox{if $c_0=0$.}
\end{cases}
\]
In particular, the Robba condition holds everywhere if $\lambda \in \ZZ_p$ and nowhere otherwise. In either case, $\calN_z$ factors through the retraction of $\PP^{1,\an}_K$ onto the path from 0 to $\infty$.
\end{example}

\begin{example} \label{exa:exponential2}
Take $K = \QQ_p$, $Z = \{\infty\}$, and consider the differential equation $y'-az^{a-1} y =0$ for some positive integer $a$ not divisible by $p$ (the case $a=1$ being Example~\ref{exa:exponential}). The formal solutions of this equation are the scalar multiples of
\[
\exp(z^a - x^a).
\]
This series converges in the region where $|z^a - x^a| < p^{-1/(p-1)}$; consequently,
\[
s_1(\calN_z(x)) = \max\left\{\frac{1}{p-1} \log p + (a-1) \log |x|, \frac{1}{a(p-1)} \log p\right\}.
\]
In this case, $\calN_z$ factors through the retraction onto the path from $x_{p^{-1/(p-1)}}$ to $\infty$.
\end{example}

\begin{example}
Take $K = \CC((t))$ (so $p=0$), $Z = \{\infty\}$, and consider the differential equation
\[
y''' + z y'' + y = 0.
\]
It can be shown that 
\[
s_1(\calN_z(x)) = \max\left\{0, \log |x| \right\},
s_2(\calN_z(x)) = s_3(\calN_z(x)) = \min\left\{0, - \frac{1}{2} \log |x| \right\}.
\]
In this case, $\calN_z$ factors through the retraction onto the path from $x_1$ to $\infty$. Note that this provides an example where the slopes of $\calN_z(x)$ are not bounded below uniformly on $(\PP^1_K - Z)^{\an}$; that is, as $x$ approaches $\infty$, two linearly independent local solutions have radii of convergence growing without bound, but these local solutions do not patch together.
\end{example}

\begin{example} \label{exa:hypergeometric}
Assume $p > 2$, take $K = \QQ_p$, $Z = \{0, 1, \infty\}$, and consider the Gaussian hypergeometric differential equation 
\[
y'' + \frac{(1-2z)}{z(1-z)} y' - \frac{1}{4z(1-z)} y = 0.
\]
This example was originally studied by Dwork \cite{dwork-cycles} due to its relationship with the zeta functions of elliptic curves. Using Dwork's calculations, it can be shown that
\[
s_1(\calN_z(x)) = s_2(\calN_z(x)) = \max\{ \log |x|, -\log |x|, -\log |x-1|\}.
\]
In this case, $\calN_z$ factors through the retraction from $\PP^{1,\an}_K$
onto the union of the paths from $0$ to $\infty$ and from $1$ to $\infty$,
and the Robba condition holds everywhere.
\end{example}

\begin{remark}
One can compute additional examples of convergence polygons associated to first-order differential equations using an explicit formula for the radius of convergence
at a point, due to Christol--Pulita. This result
was originally reported in \cite{christol} but with an error in the formula;
for a corrected statement, see \cite[Introduction, Th\'eor\`eme~5]{pulita-hdr}.
\end{remark}

\section{Convergence polygons: general curves}
\label{sec:convergence general}

We now describe an analogue of the convergence polygon in a more geometric setting.
\begin{hypothesis} \label{H:convergence general}
Throughout \S\ref{sec:convergence general}, assume that $\chi(U) \leq 0$, i.e., either $g(X) \geq 1$ or $\length(Z) \geq 2$.
Let $\calE$ be a vector bundle on $U$ of rank $n$ equipped with a
connection $\nabla: \calE \to \calE \otimes_{\calO_U} \Omega_{U/K}$.
As is typical, we describe sections of $\calE$ in the kernel of $\nabla$ as being \emph{horizontal}.
\end{hypothesis}

\begin{remark}
For the results in this section, one could also allow $X$ to be an analytic curve which is compact but not necessarily proper. To simplify the discussion, we omit this level of generality; the general results can be found in any of \cite{baldassarri-kedlaya}, \cite{kedlaya-radii}, \cite{pulita-poineau2}.
\end{remark}

\begin{defn} \label{D:convergence polygon general}
Let $L$ be a nonarchimedean field containing $K$ and choose $x \in U(L)$, which we identify canonically with a point of $U_L^{\an}$.
Since $X$ is smooth, $U_{L}^{\an}$ contains a neighborhood of $x$ isomorphic to an open disc over $L$. Thanks to our restrictions on $X$ and $Z$, the union $U_x$ of all such neighborhoods in $U_L^{\an}$ is itself isomorphic to an open disc over $L$.
The construction is compatible with base change in the following sense: if $L'$ is a nonarchimedean field containing $L$, then $U_{x,L'}$ is the union of all neighborhoods of $x$ in $U_{L'}^{\an}$ isomorphic to an open disc over $L'$. 
For each $\rho \in (0,1]$, let $U_{x,\rho}$ be the open disc of radius $\rho$ centered at $x$ within $U_x$ (normalized so that $U_{x,1} = U_x$).

Let $\widehat{\calO}_{X_L,x}$ denote the completed local ring of $X_L$ at $x$; it is abstractly a power series ring in one variable over $L$. 
Let $\calE_x$ denote the pullback of $\calE$ to $\widehat{\calO}_{X_L,x}$,
equipped with the induced connection. One checks easily that $\calE_x$ is a trivial differential module; more precisely, the space $\ker(\nabla, \calE_x)$ is an $n$-dimensional vector space over $L$ and the natural map
\[
\ker(\nabla, \calE_x) \otimes_{L} \widehat{\calO}_{X_L,x} \to \calE_x
\]
is an isomorphism.

For $i=1,\dots,n$, choose $s_i(x) \in [0, +\infty)$ so that $e^{-s_i(x)}$ is the supremum of the set of $\rho \in (0,1]$ such that $\calE_x$ contains $n-i+1$ linearly independent 
sections convergent on $U_{x,\rho}$. Again, this set of such $\rho$ is nonempty by
Theorem~\ref{T:p-adic Cauchy}. Since $s_1(x) \geq \cdots \geq s_n(x)$, the $s_i(x)$ are the slopes of a polygon $\calN(x) \in \calNP[0,n]$, which we call the
\emph{convergence polygon} of $\calE$ at $x$. Again, the construction is compatible with base change, so it induces a well-defined function $\calN: U^{\an} \to \calNP[0,n]$.
\end{defn}

\begin{defn}
For $x \in U^{\an}$, we say that $\calE$ satisfies the \emph{Robba condition} at $x$
if $\calN(x)$ is the zero polygon.
\end{defn}

We have the following analogue of Theorem~\ref{T:continuous1}.
\begin{theorem} \label{T:continuous2}
The function $\calN: (X-Z)^{\an} \to \calNP[0,n]$ is continuous. More precisely, there exists a strict skeleton $\Gamma$ such that $\calN$ factors through the retraction of $X^{\an}$ onto $\Gamma$, and the restriction of $\calN$ to each edge of $\Gamma$ is affine with integral derivative.
\end{theorem}
\begin{proof}
See \cite{pulita-poineau2} or \cite{kedlaya-radii} or \cite{baldassarri-kedlaya}
(and Remark~\ref{R:continuous2}).
\end{proof}

\begin{remark} \label{R:continuous2}
It is slightly inaccurate to attribute Theorem~\ref{T:continuous2} to \cite{pulita-poineau2} or \cite{baldassarri-kedlaya}, as the results proved therein are 
slightly weaker: they require an uncontrolled base extension on $K$, which creates more options for the strict skeleton $\Gamma$.
In particular, Theorem~\ref{T:continuous2} as stated implies that $\calN$ is locally constant around any point of type 4, which cannot be established using the methods of \cite{pulita-poineau2} or \cite{baldassarri-kedlaya}; one instead requires some dedicated arguments found only in \cite{kedlaya-radii}. These extra arguments are crucial for the applications of Theorem~\ref{T:continuous2} in the contexts described in
\cite{kedlaya-goodformal2} and \cite{kedlaya-semistable4}.
\end{remark}

\begin{remark} \label{R:p1 case}
Suppose that $X = \PP^1_K$ and $\infty \in Z$. Given a differential equation as in 
\eqref{eq:de}, we can construct an associated connection $\calE$ of rank $n$ whose underlying vector bundle is free on the basis $\be_1,\dots,\be_n$ and whose action of $\nabla$ is given by
\begin{align*}
\nabla(\be_1) &= f_0(z) \be_n \\
\nabla(\be_2) &= f_1(z)\be_n -\be_1\\
\vdots \\
\nabla(\be_{n-1}) &= f_{n-2}(z) \be_n - \be_{n-2} \\
\nabla(\be_n) &= f_{n-1}(z) \be_n - \be_{n-1}.
\end{align*}
A section of $\calE$ is then horizontal if and only if it has the form
$y \be_1 + y' \be_2 + \cdots + y^{(n-1)} \be_n$ where $y$ is a solution of \eqref{eq:de}. 
If $\length(Z) \geq 2$, each of $\calN_z$ and $\calN$ can be computed in terms of the other; this amounts to changing the normalization of certain discs. In particular,
the statements of Theorem~\ref{T:continuous1} and Theorem~\ref{T:continuous2} in this case are equivalent.

If $\length(Z) = 1$, we cannot define $\calN$ as above. However, 
if $K$ is nontrivially valued,
one can recover the properties of $\calN_z$ by considering $\calN$ with $Z$ replaced by $Z \cup \{x\}$ for some $x \in U(K)$ with $|x|$ sufficiently large (namely, larger than the radius of convergence of any nonentire formal solution at $0$). We refer to
\cite{baldassarri-kedlaya} for further details.
\end{remark}

\begin{remark} \label{R:p1 connection}
One can extend Remark~\ref{R:p1 case} by defining $\calN_z$ in the case where
$X = \PP^1_K$ and $\infty \in Z$, and using Theorem~\ref{T:continuous2} to establish an analogue of Theorem~\ref{T:continuous1}.
With this modification, we still do not define either $\calN$ or $\calN_z$ in the case where $X = \PP^1_K$ and $Z = \emptyset$, but this case is completely trivial: the vector bundle $\calE$ must admit a basis of horizontal sections
(see \cite{baldassarri-kedlaya}).
\end{remark}

\begin{theorem} \label{T:transfer}
Suppose that $x \in \Gamma \cap 
U^{\an}$ is the generic point of a open disc $D$ contained in $X$
and the Robba condition holds at $x$.
\begin{enumerate}
\item[(a)]
If $D \cap Z = \emptyset$, then the restriction of $\calE$ to $D$ is trivial (i.e., it admits a basis of horizontal sections).
\item[(b)]
If $D \cap Z$ consists of a single point $z$ at which $\nabla$ is regular, then
the Robba condition holds on $D - \{z\}$.
\end{enumerate}
\end{theorem}
\begin{proof}
Part (a) is a special case of the Dwork transfer theorem; see for instance
\cite[Theorem~9.6.1]{kedlaya-book}.
Part (b) follows as in the proof of \cite[Theorem~13.7.1]{kedlaya-book}.
\end{proof}

\begin{remark} \label{R:p-adic Liouville numbers}
Theorem~\ref{T:transfer}(b) is a variant of a result of Christol, which has a slightly stronger hypothesis and a slightly stronger conclusion. In Christol's result, one must assume either that $p=0$, or that $p>0$ and the pairwise differences between the exponents of $\nabla$ at $z$ are not $p$-adic Liouville numbers (see Example~\ref{exa:Liouville}). One however gets the stronger conclusion that the ``formal solution matrix'' of $\nabla$ at $z$ converges on all of $D$. Both parts of Theorem~\ref{T:transfer} are examples of
\emph{transfer theorems}, which can be viewed as transferring convergence information from one disc to another.
\end{remark}

\begin{remark} \label{R:parameters}
Let $k$ be the residue field of $K$.
Suppose that $X = \PP^1_K$, $\nabla$ is regular everywhere,
and the reduction map from $Z$ to $\PP^1_k$ is injective. Then
Theorem~\ref{T:transfer} implies that if the Robba condition holds at $x_1$, then it holds on all of $U^{\an}$. For instance, this is the case for (the connection
associated via Remark~\ref{R:p1 case} to) the
hypergeometric equation 
considered in Example~\ref{exa:hypergeometric}; more generally, it holds for the
hypergeometric equation
\[
y'' - \frac{c-(a+b+1)z}{z(1-z)} y' - \frac{ab}{z(1-z)} y = 0
\]
if and only if $a,b,c \in \ZZ_p$ (the case of Example~\ref{exa:hypergeometric} being
$a=b=1/2, c=1$). 
This example and Example~\ref{exa:logarithmic parameter}, taken together, suggest that
for a general differential equation with one or more accessory parameters, the Robba condition at a fixed point is likely to be of a ``fractal'' nature in these parameters.
For some additional examples with four singular points, see the work of Beukers
\cite{beukers}.
\end{remark}

\begin{remark}
One can also consider some modified versions of the convergence polygon. For instance, one might take $e^{-s_i(x)}$ to be the supremum of those $\rho \in (0,1]$ such that
the restriction of $\calE$ to $U_{x,\rho}$ splits off a trivial submodule of rank at least $n-i+1$; the resulting convergence polygons will again satisfy Theorem~\ref{T:continuous2}. It may be that some modification of this kind can be used to eliminate some hypotheses on $p$-adic exponents, as in
Theorem~\ref{T:virtual local index}.
\end{remark}

\section{Derivatives of convergence polygons}
\label{sec:slopes}

We now take a closer look at the local variation of convergence polygons.
Throughout \S\ref{sec:slopes}, continue to retain Hypothesis~\ref{H:convergence general}.

\begin{defn} \label{D:branch}
For $x \in X^{\an}$,
a \emph{branch} of $X$ at $x$ is a local connected component of $X - \{x\}$, that is,
an element of the direct limit of $\pi_0(U - \{x\})$ as $U$ runs over all neighborhoods of $x$ in $X$. Depending on the type of $x$, the branches of $X$ can be described as follows.
\begin{enumerate}
\item[Type 1:]
A single branch.
\item[Type 2:]
One branch corresponding to each closed point on the curve $C_x$ (defined over the residue field of $K$) whose function field is the residue field of $\calH(x)$.
\item[Type 3:]
Two branches.
\item[Type 4:]
One branch.
\end{enumerate}
For each branch $\vec{t}$ of $X$ at $x$, by Theorem~\ref{T:continuous2} we may define the \emph{derivative} of $\calN$ along $\vec{t}$ (away from $x$), as an element of $\calP[0,n]$ with integral vertices;
we denote this element by $\partial_{\vec{t}}(\calN)$.
For $x$ of type 1, we also denote this element by $\partial_x(\calN)$ since there is no ambiguity about the choice of the branch.
We may similarly define $\partial_{\vec{t}}(h_i(\calN)) \in \RR$ for $i=1,\dots,n$,
optionally omitting $i$ in the case $i=n$; note that $\partial_{\vec{t}}(h(\calN)) \in \ZZ$.
\end{defn}

\begin{theorem} \label{T:Turrittin}
For $z \in Z$, $-\partial_{z}(\calN)$ is the polygon associated to the 
Turrittin-Levelt-Hukuhara decomposition of $\calE_z$
(see \cite[Chapter~4]{dgs}, \cite[\S 11]{katz-turrittin}, or \cite[Chapter~7]{kedlaya-book}). In particular,
this polygon belongs to $\calNP[0,n]$,
its slopes are all nonnegative,
and its height equals the irregularity $\Irr_z(\nabla)$ of $\nabla$ at $z$.
\end{theorem}
\begin{proof}
See \cite[\S 5.7]{pulita-poineau3} and \cite[\S 3.6]{pulita-poineau4}, or see \cite{baldassarri-kedlaya}. 
\end{proof}

\begin{cor}
For $z \in Z$, $\calN$ extends continuously to a neighborhood of $z$
if and only if $\nabla$ has a regular singularity at $z$ (i.e., its irregularity at $z$ equals $0$).
In particular, $\calN$ extends continuously to all of $X^{\an}$ if and only if $\nabla$ is everywhere regular.
\end{cor}

\begin{remark} \label{R:asympotic irregularity}
Using a similar technique, one can compute the asymptotic behavior of $\calN$ in a neighborhood of $z \in Z$ in terms of the ``eigenvalues'' occurring in the 
Turrittin-Levelt-Hukuhara decomposition of $\calE_z$.
For example, $\nabla$ satisfies the Robba condition on some neighborhood of $z$ if and only if $\nabla$ is regular at $z$ with all exponents in $\ZZ_p$.
\end{remark}

\begin{remark}
In the complex-analytic setting, the decomposition of $\calE_z$ does not typically extend to any nonformal neighborhood of $z$; this is related to the discrepancy between local indices in the algebraic and analytic categories (see Remark~\ref{R:index conditions}).
Nonetheless, there are still some close links between the formal decomposition and the asymptotic behavior of local solutions on sectors at $z$ (related to the theory of \emph{Stokes phenomena}).

In the nonarchimedean setting, by contrast, the decomposition of $\calE_z$ always lifts to some nonformal neighborhood of $z$; this follows from a theorem of Clark \cite{clark}, which asserts that formal solutions of a connection at a (possibly irregular) singular point always converge in some punctured disc.
\end{remark}

\begin{theorem} \label{T:monotonicity}
Assume $p=0$.
For $x \in U^{\an}$ and $\vec{t}$ a branch of $X$ at $x$ not pointing along $\Gamma_{X,Z}$, we have $\partial_{\vec{t}}(\calN) \leq 0$.
\end{theorem}
\begin{proof}
This requires somewhat technical arguments not present in the existing literature;
Theorem~\ref{T:monotonicity appendix2}.
\end{proof}

\begin{remark}
In the setting of Theorem~\ref{T:monotonicity}, the statement that $\partial_{\vec{t}}(h_1(\calN)) \leq 0$ is equivalent to the Dwork transfer theorem
(again see \cite[Theorem~9.6.1]{kedlaya-book}).
In general, Theorem~\ref{T:monotonicity} is deduced by relating $\partial_{\vec{t}}(\calN)$
to local indices, as discussed in \S\ref{sec:subharmonicity}.
This argument cannot work for $p>0$ due to certain pathologies related to $p$-adic Liouville numbers (see Remark~\ref{R:star2}). We are hopeful that one can use a perturbation argument to deduce the analogue of Theorem~\ref{T:monotonicity} for $p>0$, by reducing to the case where $\calN(x)$ has no slopes equal to 0 and applying results of \cite{kedlaya-book} (especially \cite[Theorem~11.3.4]{kedlaya-book}); however, a proof along these lines was not ready at the time of this writing.
\end{remark}

\begin{remark} \label{R:branches slopes}
By Theorem~\ref{T:continuous2}, for each $x \in U^{\an}$, there exist only finitely many branches $\vec{t}$ at $x$ along which $\calN$ has nonzero slope. If $x$ is of type 1 or 4, there are in fact no such branches. If $x$ is of type 3, then the slopes along the two branches at $x$ add up to 0.
\end{remark}

\section{Subharmonicity and index}
\label{sec:subharmonicity}

Using the piecewise affine structure of the convergence polygon, we formulate some additional properties, including local and global index formulas for de Rham cohomology.
The local index formula is due to Poineau and Pulita \cite{pulita-poineau5},
generalizing some partial results due to Robba \cite{robba-index1, robba-index2,
robba-index3, robba-indice4}
and Christol-Mebkhout \cite{cm1, cm2, cm3, cm4}.
Unfortunately, in the case $p>0$ one is forced to interact with a fundamental pathology in the theory of $p$-adic differential equations, namely the effect of \emph{$p$-adic Liouville numbers}; consequently, the global formula we derive here cannot be directly  deduced from the local formula (see Remark~\ref{R:star} and Remark~\ref{R:star2}).

\begin{hypothesis}
Throughout \S\ref{sec:subharmonicity}, continue to retain Hypothesis~\ref{H:convergence general}, but assume in addition that $K$ is algebraically closed. (Without this assumption, one can still formulate the results at the expense of having to keep track of some additional multiplicity factors.)
\end{hypothesis}

\begin{defn}
For $x \in U^{\an}$, let $(\Delta \calN)_x \in \calP[0,n]$ denote the sum of $\partial_{\vec{t}}(\calN)$ over all branches $\vec{t}$ of $X$ at $x$ (oriented away from $x$); by Remark~\ref{R:branches slopes}, this sum can only be nonzero when $x$ is of type 2. 
Define the \emph{Laplacian} of $\calN$ as the $\calP[0,n]$-valued measure $\Delta \calN$ taking a continuous function $f: U^{\an} \to \RR$
to $\sum_{x \in U^{\an}} f(x) (\Delta \calN)_x$. For $i=1,\dots,n$, we may similarly define the multiplicities $(\Delta h_i(\calN))_x \in \RR$
and the Laplacian $\Delta h_i(\calN)$; we again omit the index $i$ when it equals $n$.
\end{defn}

\begin{remark}
The definition of the Laplacian 
can also be interpreted in the context of Thuillier's potential theory
\cite{thuillier}, which applies more generally to functions which need not be piecewise affine.
\end{remark}

\begin{lemma} \label{L:laplacian irregularity}
We have
\[
\int \Delta h(\calN) = \sum_{z \in Z} \Irr_z(\nabla).
\]
\end{lemma}
\begin{proof}
For $e$ an edge of $\Gamma$, we may compute the slopes of $\calN$ along the two branches pointing into $e$ from the endpoints of $e$; these two slopes add up to 0. If we add up these slopes over all $e$, then regroup this sum by vertices, then the sum
at each vertex $z \in Z$ equals $-\Irr_z(\nabla)$ by Theorem~\ref{T:Turrittin}, while the sum at each vertex $x \in U^{\an}$ is the multiplicity of $x$ in $\Delta \calN$. This proves the claim.
\end{proof}

\begin{defn}
For any open subset $V$ of $X^{\an}$,
consider the complex
\[
0 \to \calE \stackrel{\nabla}{\to} \calE \otimes \Omega \to 0
\]
of sheaves, keeping in mind that if $V \cap Z \neq \emptyset$, then the sections over $V$ are allowed to be meromorphic at $V \cap Z$ (but not to have essential singularities;
see Remark~\ref{R:index conditions}).
We define $\chi_{\dR}(V, \calE)$ to be the index of the hypercohomology of this complex,
i.e., the alternating sum of $K$-dimensions of the hypercohomology groups.
\end{defn}

\begin{lemma} \label{L:index formula from irregularity}
We have
\begin{equation} \label{eq:index formula from irregularity}
\chi_{\dR}(X^{\an}, \calE) 
= n \chi(U) - \sum_{z \in Z} \Irr_z(\nabla) = n (2 - 2g(X) - \length(Z)) - \sum_{z \in Z} \Irr_z(\nabla).
\end{equation}
\end{lemma}
\begin{proof}
Let $K_0$ be a subfield of $K$ which is finitely generated over $\QQ$
to which $X, Z,\calE, \nabla$ can be descended.
Then choose an embedding $K_0 \subset \CC$ and let $X_\CC$ be the base extension of the descent of $X$, again equipped with a meromorphic vector bundle $\calE$ and connection $\nabla$.
Note that $\chi_{\dR}(X^{\an}, \calE)$ is computed by a spectral sequence in which one first computes the coherent cohomology of $\calE$ and $\calE \otimes \Omega$ separately.
By the GAGA principle both over $\CC$ 
\cite[Expos\'e XII]{sga1} and $K$ \cite[Example~3.2.6]{conrad},
these coherent cohomology groups can be computed equally well over any of $X^{\an}$, $X$
(or its descent to $K_0$), $X_{\CC}$, or $X_{\CC}^{\an}$. Consequently,
despite the fact that the connection is only $K$-linear rather than $\calO$-linear, we may nonetheless conclude that
$\chi_{\dR}(X^{\an}, \calE) = \chi_{\dR}(X_{\CC}^{\an}, \calE)$.
(As an aside, this argument recovers a comparison theorem of Baldassarri \cite{baldassarri-comparison}.)

To compute $\chi_{\dR}(X_{\CC}^{\an},\calE)$, we may either appeal directly to
\cite[(6.21.1)]{deligne-eq} or argue directly as follows. Form a finite open covering $\{V_i\}_{i \in I}$ of $X_{\CC}$ such that for each nonempty subset $S$ of $I$, the set $V_S = \bigcap_{i \in S} V_i$ satisfies the following conditions.
\begin{itemize}
\item
If nonempty, $V_S$ is isomorphic to a simply connected domain in $\CC$.
\item
The set $V_S \cap Z$ contains at most one element.
\end{itemize}
We then have
\[
\chi_{\dR}(X^{\an}_{\CC}, \calE) = 
\sum_{S \subseteq I, S \neq \emptyset}
(-1)^{\#S-1}
\chi_{\dR}\left(V_S, \calE \right).
\]
It then suffices to check that for each nonempty subset $S$ of $I$,
\begin{equation} \label{eq:deligne-malgrange}
\chi_{\dR}(V_S, \calE) = \begin{cases} -\Irr_z(\nabla) & (V_S \cap Z = \{z\}) \\
n & (V_S \cap Z = \emptyset, V_S \neq \emptyset) \\
0 & (V_S= \emptyset).
\end{cases}
\end{equation}
In case $V_S = \emptyset$, there is nothing to check.
In case $V_S \neq \emptyset$ but $V_S \cap Z = \emptyset$, this is immediate because the restriction of $\calE$ to $V_S$ is trivial. In case $V_S \cap Z = \{z\}$,
we may similarly replace $V_S$ with a small open disc around $z$,
and then invoke the Deligne-Malgrange interpretation of irregularity as the local index of meromorphic de Rham cohomology on a punctured disc \cite[Th\'eor\`eme~3.3(d)]{malgrange}.
\end{proof}

\begin{theorem}[Global index formula] \label{T:index formula}
We have
\begin{equation} \label{eq:index formula}
\chi_{\dR}(X^{\an}, \calE) 
= n \chi(U) - \int \Delta h(\calN) = n (2 - 2g(X) - \length(Z)) - \int \Delta h(\calN).
\end{equation}
\end{theorem}
\begin{proof}
This follows by comparing Lemma~\ref{L:laplacian irregularity} with
Lemma~\ref{L:index formula from irregularity}.
\end{proof}

\begin{remark}  \label{R:index conditions}
It may not be immediately obvious why Theorem~\ref{T:index formula} is of value, i.e.,
why it is useful to express the index of de Rham cohomology in terms of convergence polygons instead of irregularity. As observed by Baldassarri \cite[(0.13)]{baldassarri-comparison}, there is a profound difference between the behavior of the index in the complex-analytic and nonarchimedean settings. In the complex case, for any open analytic subspace $V$ of $X^{\an}_{\CC}$, we have
\begin{equation} \label{eq:local index formula complex}
\chi_{\dR}(V, \calE) = n \chi(V \cap U^{\an}) - \sum_{z \in V \cap Z} \Irr_z(\nabla)
\end{equation}
by the same argument as in the proof of \eqref{eq:deligne-malgrange}.
In particular, $\chi_{\dR}(V, \calE) \neq \chi_{\dR}(V - Z, \calE)$; that is, the index of de Rham hypercohomology depends on whether we allow poles or essential singularities at the points of $Z$. By contrast, in the nonarchimedean case, these two indices coincide
under a suitable technical hypothesis to ensure that they are both defined; 
see Example~\ref{exa:nonarchimedean local index} for a simple example
and Corollary~\ref{C:punctures} for the general case
(and Example~\ref{exa:Liouville} and Remark~\ref{R:Liouville}
for a counterexample failing the technical hypothesis).
This means that in the nonarchimedean case, the ``source'' of the index of de Rham cohomology is not irregularity, but rather the Laplacian of the convergence polygon (see Theorem~\ref{T:local index formula}).
\end{remark}

\begin{example} \label{exa:nonarchimedean local index}
Consider the connection associated to Example~\ref{exa:add pole}
as per Remark~\ref{R:p1 case}; 
note that $\Irr_0(\nabla) = 0$, $\Irr_{\infty}(\nabla) = 1$,
so $\chi_{\dR}(X^{\an}, \calE) = -1$.
For $0 < \alpha < \beta$, let $V_\beta, W_\alpha$ be the subspace $|z| < \beta$,
$|z| > \alpha$ of $X^{\an}$. By Mayer-Vietoris,
\begin{equation} \label{eq:mayer-vietoris}
\chi_{\dR}(V_\beta, \calE) +
\chi_{\dR}(W_\alpha, \calE) -
\chi_{\dR}(V_\beta \cap W_\alpha, \calE) =
\chi_{\dR}(X^{\an}, \calE) = - 1.
\end{equation}
On the other hand, $\chi_{\dR}(V_\beta, \calE)$ (resp.\
$\chi_{\dR}(W_\alpha, \calE)$)
equals the index of the operator $y \mapsto y' - y$ on 
Laurent series in $z$ convergent for $|z| < \beta$
(resp.\ in $z^{-1}$ convergent for $|z^{-1}| < \alpha^{-1}$). 
If $f = \sum_n f_n z^{-n}$, $g = \sum_n g_n z^{-n}$ are two such series, then the equation
$g = f' - f$ is equivalent to
\begin{equation} \label{eq:nonarchimedean local index}
f_n = -g_n - (n-1) f_{n-1} \qquad (n \in \ZZ).
\end{equation}
Suppose first that $p^{-1/(p-1)} < \alpha < \beta$.
Then given $g \in K((z^{-1}))$, we may solve uniquely for $f \in K((z^{-1}))$,
and if $g$ converges on $W_\alpha - \{\infty\}$, then so does $f$. We thus compute that
\[
\chi_{\dR}(V_\beta, \calE) = -1, \quad
\chi_{\dR}(W_\alpha, \calE) = 0, \quad
\chi_{\dR}(V_\beta \cap W_\alpha, \calE) = 0;
\]
namely, the second equality is what we just computed, the third
follows from Robba's index formula \cite{robba-indice4} 
(see also \cite[Lemma~3.7.5]{kedlaya-radii}), 
and the first follows from the other two plus \eqref{eq:mayer-vietoris}.
In particular, the local index at $\infty$ equals 0, whereas in the complex-analytic setting it equals $-1$ by the Deligne-Malgrange formula
(see the proof of Lemma~\ref{L:index formula from irregularity}).

Suppose next that $\alpha < \beta < p^{-1/(p-1)}$. Then by contrast, we have
\[
\chi_{\dR}(V_\beta, \calE) = 0, \quad
\chi_{\dR}(W_\alpha, \calE) = -1, \quad
\chi_{\dR}(V_\beta \cap W_\alpha, \calE) = 0;
\]
namely, the first and third equalities follow from the triviality of $\nabla$ 
on $V_\beta$, and the second follows from the other two plus \eqref{eq:mayer-vietoris}.
\end{example}

With Example~\ref{exa:nonarchimedean local index} in mind, we now describe a local refinement of Theorem~\ref{T:index formula}, in which we dissect the combinatorial formula for the index into local contributions.

\begin{defn}
For $x \in U^{\an} \cap \Gamma_{X, Z}$, let $\val_\Gamma(x)$ be the valence of $x$ as a vertex of $\Gamma_{X,Z}$,
taking $\val_{\Gamma}(x) = 2$ when $x$ lies on the interior of an edge.
(We refer to \emph{valence} instead of \emph{degree} to avoid confusion with degrees of morphisms.)
For $x \in U^{\an}$,
define 
\begin{equation} \label{eq:virtual local index}
\chi_{x}(\calE) = \begin{cases} n(2 - 2g(C_x) - \val_{\Gamma}(x)) - (\Delta h(\calN))_x &\mbox{if $x \in \Gamma_{X,Z}$} \\
-\Delta h(\calN)_x & \mbox{otherwise.}
\end{cases}
\end{equation}
Let $\chi(\calE)$ be the $\RR$-valued measure whose value on a continuous function $f: U^{\an} \to \RR$ is $\sum_{x \in U^{\an}} f(x) \chi_{x}(\calE)$.
\end{defn}

\begin{lemma} \label{L:sum of virtual local indices}
We have
\[
\chi_{\dR}(X^{\an}, \calE) = n \chi(U) - \int \Delta h(\calN) = \int \chi(\calE).
\]
\end{lemma}
\begin{proof}
This follows from Theorem~\ref{T:index formula} plus the identity
\[
\sum_{x \in U^{\an} \cap \Gamma_{X,Z}} (2 - 2g(C_x) - \val_{\Gamma}(x))
= 2 - 2g(X) - \length(Z),
\]
which amounts to the combinatorial formula for the genus of an analytic curve
\cite[\S 4.16]{bpr}.
\end{proof}

\begin{theorem}[Local index formula] \label{T:local index formula}
Let $V$ be an open subspace of $X^{\an}$
which is the retraction of an open subspace of $\Gamma_{X,Z}$.
If $p>0$, assume some additional technical hypotheses (see Remark~\ref{R:star}). Then
\[
\chi_{\dR}(V, \calE) = \int_{V \cap U^{\an}} \chi(\calE).
\]
\end{theorem}
\begin{proof}
See \cite[Theorem~3.5.2]{pulita-poineau5}.
\end{proof}

\begin{cor} \label{C:punctures}
With hypotheses as in Theorem~\ref{T:local index formula},
$\chi_{\dR}(V, \calE) = \chi_{\dR}(V \cap U^{\an}, \calE)$; that is, the index of de Rham hypercohomology is the same whether we allow poles or essential singularities at $Z$.
\end{cor}

\begin{remark} \label{R:star}
Let $\Gamma$ be a strict skeleton for which the conclusion of Theorem~\ref{T:continuous2} holds.
For $v$ a vertex of $\Gamma$, define the \emph{star} of $v$, denoted $\star_v$,
as the union of $v$ and the interiors of the edges of $\Gamma$ incident to $v$.
Let $\pi_\Gamma: X \to \Gamma$ be the retraction onto $\Gamma$, through which $\calN$ factors. Under the hypotheses of Theorem~\ref{T:local index formula}, we have
\begin{equation} \label{eq:local index contributions}
\chi_{\dR}(\pi_\Gamma^{-1}(\star_v), \calE) = \chi_{v}(\calE), \qquad
\chi_{\dR}(\pi_\Gamma^{-1}(\star_v \cap \star_w), \calE) = 0.
\end{equation}
We can then recover Theorem~\ref{T:index formula} from \eqref{eq:local index contributions} by using Mayer-Vietoris (and GAGA over $K$; see
Remark~\ref{R:index conditions})
to write
\[
\chi_{\dR}(X^{\an}, \calE) = \sum_v \chi_{\dR}(\pi_\Gamma^{-1}(\star_v), \calE)
- \sum_{v \neq w} \chi_{\dR}(\pi_\Gamma^{-1}(\star_v \cap \star_w), \calE);
\]
one may similarly deduce Theorem~\ref{T:local index formula} from \eqref{eq:local index contributions}.
\end{remark}

\begin{remark} \label{R:enlarge Z for index}
For a given connection, one can extend the range of applicability of Theorem~\ref{T:local index formula} by enlarging the set $Z$; however, this depends on an understanding of how the two sides of the formula depend on $Z$. To this end, let us rewrite the equality as
\[
\chi_{\dR}(V, \calE, Z) = \int_{V \cap U^{\an}} \chi(\calE, Z)
\]
with $Z$ included in the notation.

Put $Z' = Z \cup \{z'\}$ for some $z' \in U(K)$.
Let $x' \in \Gamma_{X,Z}$ be the generic point of the open disc $U_{z'}$; by subdividing if necessary, we may view $x'$ as a vertex of $\Gamma_{X,Z}$. Then $\Gamma_{X,Z'}$ is the union of $\Gamma_{X,Z}$ with a single edge joining $x'$ to $z'$ within $U_{z'}$. 

Let $V$ be an open subset of $X^{\an}$ containing $z'$ which is the retraction of an open subspace of $\Gamma_{X,Z}$; it is then also a retraction of an open subspace of $\Gamma_{X,Z'}$.
We then have
\[
\chi_{\dR}(V, \calE, Z') - \chi_{\dR}(V, \calE, Z) = - n;
\]
namely, by Mayer-Vietoris this reduces to the case where $V$ is a small open disc around $z'$, in which case we may assume $\calE$ is trivial on $V$ and make the computational directly.

By this computation, Theorem~\ref{T:local index formula}, and the fact that $\chi_{z'}(\calE,Z) = 0$, we must also have
\[
\int_{(V \cap U^{\an}) \setminus \{z'\}} \left( \chi(\calE, Z') - \chi(\calE,Z) \right) = -n.
\]
From \eqref{eq:virtual local index}, we see that $\chi(\calE, Z') - \chi(\calE,Z)$ is supported within $U_{z'} \cup \{x'\}$. However, while one can easily compute the convergence polygon associated to $Z'$ from the one associated to $Z$ (see Example~\ref{exa:add pole} and Remark~\ref{R:cyclic change Z} for examples), we do not know how to predict \emph{a priori} where the support of $\chi(\calE, Z') - \chi(\calE,Z)$ will lie within $U_{z'} \cup \{x'\}$.
\end{remark}

\begin{remark}
One of the main reasons we have restricted attention to meromorphic connections on proper curves is that in this setting, Theorem~\ref{T:continuous2} ensures that $\chi_{x}(\calE) = 0$ for all but finitely many $x \in U^{\an}$. It is ultimately more natural to state Theorem~\ref{T:local index formula} for connections on open analytic curves, as is done in \cite[Theorem~3.8.10]{pulita-poineau5}; however, this requires some additional hypotheses to ensure that $\chi(\calE)$ is a finite measure.
\end{remark}

\begin{defn}
Assume $p>0$. 
A \emph{$p$-adic Liouville number} is an element $x \in \ZZ_p - \ZZ$ such that
\[
\liminf_{m \to \infty} \left\{ \frac{|y|}{m}: y \in \ZZ, y-x \in p^m \ZZ_p \right\} < +\infty.
\]
As in the classical case, $p$-adic Liouville numbers are always transcendental
\cite[Proposition VI.1.1]{dgs}.
\end{defn}

\begin{remark} \label{R:star2}
Assume $p>0$. The technical hypotheses of Theorem~\ref{T:local index formula}
are needed to guarantee the existence of the indices appearing in
\eqref{eq:local index contributions}. In case $\nabla$ has a regular singularity at $z \in Z$ with all exponents in $\ZZ_p$, these hypotheses include the condition that no two exponents of $\nabla$ at $z$ differ by a $p$-adic Liouville number; see Example~\ref{exa:Liouville} for a demonstration of the necessity of such a condition.
(Such hypotheses are not needed in Theorem~\ref{T:index formula} because there we only poles rather than essential singularities.)

Unfortunately, the full hypotheses are somewhat more complicated to state. They arise from the fact that with notation as in Remark~\ref{R:star}, one can separate off a maximal component of $\calE$ on $\pi_\Gamma^{-1}(\star_v \cap \star_w)$ which satisfies the Robba condition, to which one may associate some $p$-adic numbers playing the role of exponents; the hypothesis is that (for any particular $v,w$) no two of these numbers differ by a $p$-adic Liouville numbers. The difficulty is that the definition of these \emph{$p$-adic exponents}, due to Christol and Mebkhout (and later simplified by Dwork) is somewhat indirect; they occur as ``resonant frequencies'' for a certain action by the group of $p$-power roots of unity, which are hard to control except in some isolated cases where they are forced to be rational numbers (e.g., Picard-Fuchs equations, a/k/a Gauss-Manin connections, or connections arising from $F$-isocrystals in the theory of crystalline cohomology). See \cite[Chapter~13]{kedlaya-book} for more discussion.
\end{remark}

\begin{example} \label{exa:Liouville}
Assume $p>0$ and take $X = \PP^1_K$, $Z = \{0, \infty\}$. Take $\calE$ to be free of rank 1 with the action of $\nabla$ given by
\[
\nabla (f) = \lambda f \frac{dz}{z} + df
\]
for some $\lambda \in K$.
For $\alpha, \beta$ with $ 0< \alpha < \beta$, let $V$ be the open annulus $\alpha < |z| < \beta$. The 1-forms on $V$ are series
$\sum_{n=-\infty}^\infty c_n z^n \frac{dz}{z}$ 
such that for each $\rho \in (\alpha,\beta)$, $|c_n| \rho^n \to 0$ as $n \to \pm \infty$.

If $\lambda = 0$, then $|c_n| \rho^n \to 0$ for all $\rho \in (\alpha, \beta)$ if and only if $|c_n/n| \rho^n \to 0$ for all $\rho \in (\alpha,\beta)$, so every 1-form on $V$ with $c_0 = 0$ is in the image of $\nabla$.
Note that multiplying the generator of $\calE$ by $t$ has the effect of replacing $\lambda$ by $\lambda+1$; it follows that if $\lambda \in \ZZ$, then the kernel and cokernel of $\nabla$ on $V$ are both 1-dimensional, so $\chi_{\dR}(V, \calE) = 0$.

If $\lambda \in K - \ZZ$, then a 1-form is in the image of $\nabla$ on $V$ if and only if for each $\rho \in (\alpha,\beta)$, $|c_n/(n-\lambda)| \rho^n \to 0$ as $n \to \pm\infty$.
This holds if $\lambda$ is not a $p$-adic Liouville number (see \cite[\S VI.1]{dgs} or \cite[Proposition~13.1.4]{kedlaya-book}); otherwise, one shows that $\nabla$ has infinite-dimensional cokernel on $V$,
so $\chi_{\dR}(V, \calE)$ is undefined.
\end{example}

\begin{remark} \label{R:Liouville}
Example~\ref{exa:Liouville} provides an example showing that the equality 
$\chi_{\dR}(V, \calE) = \chi_{\dR}(V \cap U^{\an}, \calE)$
of Corollary~\ref{C:punctures} cannot hold without conditions on $p$-adic Liouville numbers. In this example, for all $\lambda$, $\chi_{\dR}(X^{\an}, \calE) = 0$ by
Lemma~\ref{L:index formula from irregularity}; but when $\lambda$ is a $p$-adic Liouville number,  $\chi_{\dR}(U^{\an}, \calE)$ is undefined.
\end{remark}

\begin{remark} \label{R:star3}
The net result of Remark~\ref{R:star2} is that in general, one can only view $\chi_x(\calE)$ as a \emph{virtual local index} of $\calE$ at $x$, not a true local index.
Nonetheless, this interpretation can be used to predict combinatorial properties of the convergence polygon which often continue to hold even without restrictions on $p$-adic exponents. For example, Theorem~\ref{T:monotonicity} corresponds to the fact that if $V$ is an open disc in $U^{\an}$, then the dimension of the cokernel of $\nabla$ on $V$ is nonnegative; this argument appears in the proof of Theorem~\ref{T:monotonicity} in the case $p=0$ (see Theorem~\ref{T:monotonicity appendix2}).
\end{remark}

\begin{remark}
In light of Remark~\ref{R:star3}, one might hope to establish some inequalities on 
$\chi_{x}(\calE)$. One might first hope to refine Theorem~\ref{T:virtual local index} by
analogy with \eqref{eq:local index formula complex},
by proving that the measure $\Delta h(\calN)$ is nonnegative;
however, this fails already in simple examples such as Example~\ref{exa:subharmonicity}.

On the other hand, since our running hypothesis is that $\chi(U) \leq 0$, Theorem~\ref{T:index formula} and Lemma~\ref{L:sum of virtual local indices} imply that 
\[
\sum_{x \in U^{\an}} \chi_{x}(\calE) = \chi_{\dR}(X^{\an}, \calE) \leq n \chi(U) \leq 0.
\]
One might thus hope to refine Theorem~\ref{T:virtual local index} by proving that
$\chi_{x}(\calE) \leq 0$ for all $x \in U^{\an}$. Unfortunately, this is not known (and may not even be safe to conjecture) in full generality, but see Theorem~\ref{T:virtual local index} for some important special cases.
\end{remark}

\begin{theorem} \label{T:virtual local index}
Choose $x \in U^{\an}$.
\begin{enumerate}
\item[(a)] Let $R(x)$ be the infimum of the radii of open discs in $U_x$ containing $x$.
If $\calN(x)$ has no slopes equal to $-\log R(x)$, then $\chi_{x}(\calE) = 0$.
\item[(b)]
If $x \in \Gamma_{X,Z}$, 
then $\chi_x(\calE) \leq 0$.
\item[(c)]
If $p=0$, then $\chi_x(\calE) \leq 0$.
\end{enumerate}
\end{theorem}
\begin{proof}
For (a), see \cite[Proposition~6.2.11]{pulita-poineau3} or \cite{baldassarri-kedlaya}
(or reduce to the case where $\calN(x)$ has all slopes greater than $-\log R(x)$ and then apply \cite[Theorem~5.3.6]{kedlaya-radii}).
A similar argument implies (b) because in this case, $R(x) = 1$ and the zero slopes are forced to make a nonpositive contribution to the index.
For (c), see Theorem~\ref{T:convexity}.
\end{proof}

\begin{remark}
The proof of Theorem~\ref{T:virtual local index} can also be used to quantify the extent to which $\calN$ fails to factor through the retract onto $\Gamma_{X,Z}$, and hence to help identify a suitable skeleton $\Gamma$ for which the conclusion of Theorem~\ref{T:continuous2} holds.
To be precise, for $x \in \Gamma_{X,Z}$, if the restriction of $\calN$ to $\Gamma_{X,Z}$ is harmonic at $x$,
then $\calN$ is constant on the fiber at $x$ of the retraction of $X^{\an}$ onto $\Gamma_{X,Z}$. See also \cite[\S 6.3]{pulita-poineau3}.
\end{remark}

\begin{example} \label{exa:subharmonicity}
Let $h$ be a nonzero rational function on $X$, and take $Z$ to be the pole locus of $h$. Let $\calE$ be the free bundle on a single generator $\bv$ equipped with the connection 
\[
\nabla(f \bv) = \bv \otimes (df + f\,dh).
\]
For each $z \in Z$, $\Irr_z(\nabla)$ equals the multiplicity of $z$ as a pole of $h$.
By Lemma~\ref{L:sum of virtual local indices}, 
\[
\sum_{x \in U^{\an}} \chi_x(\calE) = \chi(U) - m
\]
where $m$ is the number of poles of $h$ counted with multiplicity.

Suppose now that there exists a point $x \in U^{\an}$ with $g(C_x) > 0$.
By multiplying $h$ by a suitably large element of $K$, we can ensure that $\calN(x)$ has positive slope. In this case, by Theorem~\ref{T:virtual local index} we must have 
\[
\chi_x(\calE) = 0 > 2 - 2g(C_x) - \val_{\Gamma}(x).
\]
In particular, while $\int \Delta h(\calN) = \sum_{z \in Z} \Irr_z(\nabla) \geq 0$,
the measure $\Delta h(\calN)$ is not necessarily nonnegative.
\end{example}

\section{Ramification of finite morphisms}
\label{sec:ramification}

\begin{hypothesis} \label{H:hypothesis}
Throughout \S\ref{sec:ramification},
let $f: Y \to X$ be a finite flat morphism of degree $d$ of smooth, proper, geometrically connected curves over $K$ (a nonarchimedean field of characteristic $0$),
and suppose that the restriction of $f$ to $U$ (the complement of a finite set $Z$ of closed points of $X$) is \'etale and Galois with Galois group $G$.

Let $\pi: G \to \GL(V)$ be a faithful representation of $G$ on an $n$-dimensional vector space $V$ over $K$. Note that since $G$ is finite and $\pi$ is faithful (and $K$ is of characteristic 0), the Tannakian category generated by $V$ contains all finite-dimensional $K$-linear representations of $G$, including the regular representation.
\end{hypothesis}

\begin{defn} \label{D:rep to connection}
Equip $\calE_V = \calO_Y \otimes_{K} V^\vee$ with the diagonal action of $G$ induced by the Galois action on $\calO_Y$ and the action via $\pi^\vee$ on $V^\vee$.
By faithfully flat descent, $G$-invariant sections of $\calE_V$ can be identified with sections of a vector bundle $\calE$ on $U$; moreover, the trivial connection on $\calE_V$ induces a connection $\nabla$ on $\calE$.
We are thus in the situation of Hypothesis~\ref{H:convergence general}.
\end{defn}

\begin{remark}
One way to arrive at Hypothesis~\ref{H:hypothesis} is to start with a non-Galois cover $g: W \to X$, let $f$ be the Galois closure, and let $\pi$ be the representation of $G$ induced by the trivial representation of the subgroup of $G$ fixing $W$. In this case, $\calE$ is just the pushforward of the trivial connection on $g^{-1}(U)$.
\end{remark}

\begin{remark} \label{R:tame exponent calculation}
Following up on the previous remark, note that $(\calE, \nabla)$ is always a subobject of the pushforward of the trivial connection on $f^{-1}(U)$. Since the latter is the Gauss-Manin connection associated to a smooth proper morphism (of relative dimension 0), it is everywhere regular by virtue of the Griffiths monodromy theorem \cite[Theorem~14.1]{katz-turrittin}. This implies in turn that for $z \in Z$, our original connection satisfies $\Irr_z(\nabla) = 0$ and the exponents of $\nabla$ at $z$ are rational.

To say more, let $N$ be the order of the inertia subgroup of $G$ corresponding to some element of $f^{-1}(\{z\})$; we then claim that the exponents of $\nabla$ at $z$ belong to $\frac{1}{N} \ZZ$ and generate this abelian group. To check this, we may
assume that $K$ is algebraically closed. By the Cohen structure theorem, the field $\Frac(\widehat{\calO}_{X,z})$ may be identified with $K((t))$. The vector space $\calE_z = \calE \otimes_{\calO_U} \Frac(\widehat{\calO}_{X,z})$ over this field splits compatibly with $\nabla$ as a direct sum
\[
\bigoplus_{y \in f^{-1}(\{z\})} \Frac(\widehat{\calO}_{Y,y}),
\]
each summand of which is isomorphic to $K((t^{1/N}))$. If we split this summand further
as $\bigoplus_{i=0}^{N-1} t^{i/N} K((t))$, then the $i$-th summand contributes an exponent  congruent to $i/N$ modulo $\ZZ$. 
In particular, when $\pi$ is the regular representation, the exponents of $\nabla$ at $z$ belong to $\frac{1}{N} \ZZ$ and fill out the quotient
$\frac{1}{N} \ZZ/\ZZ$. For a general $\pi$, on one hand, $\pi$ occurs as a summand of the regular representation, so the exponents of $\nabla$ at $z$ again belong to $\frac{1}{N} \ZZ$. On the other hand, the regular representation occurs as a summand of some tensor product of copies of $V$ and its dual (see Hypothesis~\ref{H:hypothesis}), so the group generated by the exponents must also fill out the quotient
$\frac{1}{N} \ZZ/\ZZ$ (although the exponents themselves need not).
\end{remark}

\begin{defn} \label{D:branch locus}
For $L$ a nonarchimedean field over $K$, 
let $\CC_L$ denote a completed algebraic closure of $L$,
and let $B_L$ be the subset of $X_L$ for which $x \in B_L$ if and only if the preimage of $x$ in $Y_{\CC_L}^{\an}$ does not consist of $d$ distinct points.
(For a demonstration of this definition, see Example~\ref{exa:cyclic}.)
\end{defn}

We have the following remark about $B_L$ echoing an earlier observation about $\Gamma_{X,Z}$ (see Definition~\ref{D:minimal strict skeleton}).
\begin{remark} \label{R:branch locus base extension}
Let $L'$ be a nonarchimedean field containing $L$. Then $B_{L'}$ is contained in the inverse image of $B_L$ under the restriction map $X_{L'}^{\an} \to X_L^{\an}$, but this containment is typically strict if $L'$ is not the completion of an algebraic extension of $L$. This is because $B_{L'}$ only contains points of types 2 or 3 (except for the points of $Z$), whereas the inverse image of $B_L$ typically also contains some points of type 1.
\end{remark}

\begin{example} \label{exa:cyclic1}
Take $K = \QQ_p(\zeta_p)$,
$X = \PP^1_K$, $Z = \{0, \infty\}$, $Y = \PP^1_K$, let $f: Y \to X$ be the map
$z \mapsto z^p$, identify $G$ with $\ZZ/p\ZZ$ so that $1 \in \ZZ/p\ZZ$ corresponds to the map $z \mapsto \zeta z$ on $f^{-1}(U)$, 
and let $\pi: G \to \GL_1(K)$ be the character taking 1 to $\zeta_p^{-1}$. 
Then $\calE$ is free on a single generator $\bv$ satisfying
$\nabla(\bv) = -\frac{1}{p} z^{-1} \bv \otimes dz$.
This is also the connection obtained from Example~\ref{exa:p-th root} by applying
Remark~\ref{R:p1 case}.

Put $\omega = p^{-1/(p-1)} = |\zeta_p - 1|$.
For each $x \in U$, 
we may define the \emph{normalized diameter} of $x$ as an element of $[0,1]$ defined as follows:
choose an extension $L$ of $K$ such that $x$ lifts to some $\tilde{x} \in U(L)$,
then take the infimum of all $\rho \in (0,1]$ such that $U_{\tilde{x},\rho}$ meets
the inverse image of $x$ in $U_L^{\an}$ (or $1$ if no such $\rho$ exists).
With this definition, for any $L$, the set $B_L$ consists of $Z$ plus all points
with normalized diameter in $[\omega^p, 1]$ (see for example \cite[Lemma~10.2.2]{kedlaya-book}).
\end{example}

\begin{lemma} \label{L:covering space}
Suppose that $L$ is algebraically closed.
Let $V$ be an open subset of $X_L^{\an}$ which is disjoint from $B_L$. Then the map 
$f^{-1}(V) \to V$ of topological spaces is a covering space map.
\end{lemma}
\begin{proof}
For each $x \in V$, the local ring of $X_L^{\an}$ at $x$ is henselian 
\cite[Theorem~2.1.5]{berkovich2}; hence for each $y \in f^{-1}(\{x\})$, we can find a neighborhood $W_y$ of $y$ in $f^{-1}(V)$ which is finite \'etale over its image in $X_L^{\an}$. Since $Y_L^{\an}$ is Hausdorff, we can shrink the $W_y$ so that they are pairwise disjoint and all have the same image $V'$ in $X_L^{\an}$.
The maps $W_y \to V'$ are then finite \'etale of some degrees adding up to $\deg(f) = d$.
Since $\#f^{-1}(\{x\}) = d$ by hypothesis, these degrees must all be equal to 1; hence each map $W_y \to V'$ is actually an isomorphism of analytic spaces, and in particular a homeomorphism of topological spaces.
\end{proof}

\begin{theorem} \label{T:ramification}
For $L$ a nonarchimedean field over $K$ and $x \in U(L)$, 
$e^{-s_1(\calN(x))}$ equals the supremum of $\rho \in (0,1]$ such that
$U_{x,\rho} \cap B_L  = \emptyset$.
\end{theorem}
\begin{proof}
We may assume without loss of generality that $L$ is itself algebraically closed.
If $U_{x,\rho} \cap B_L = \emptyset$, then by Lemma~\ref{L:covering space},
the map $f^{-1}(U_{x,\rho}) \to U_{x,\rho}$ is a covering space map of topological spaces.
Since $U_{x,\rho}$ is contractible and hence simply connected, $f^{-1}(U_{x,\rho})$ splits topologically as a disjoint union of copies of $U_{x,\rho}$. From this, it follows easily that the restriction of $\calE$ to $U_{x,\rho}$ is trivial. 

Conversely, suppose that the restriction of $\calE$ to $U_{x,\rho}$ is trivial.
Then the same holds when $V$ is replaced by any representation in the Tannakian category
generated by $V$; since $G$ is finite and $\pi$ is faithful, this includes the regular representation. That is, the pushforward of the trivial connection on $f^{-1}(U)$ has trivial restriction to $U_{x,\rho}$.

In addition to the connection, the sections of this restriction inherit from $\calO_Y$ a multiplication map, and the product of two horizontal sections is again horizontal. Consequently, the horizontal sections form a finite reduced $L$-algebra of rank $d$; since $L$ is algebraically closed, this algebra must split as a direct sum of $d$ copies of $L$. This splitting corresponds to a splitting of $f^{-1}(U_{x,\rho})$ into $d$ disjoint sets, proving that $U_{x,\rho} \cap B_L = \emptyset$.
\end{proof}

\begin{remark}
One may view Theorem~\ref{T:ramification} as saying that under a suitable normalization, $e^{-s_1(\calN(x))}$ measures the distance from $x$ to $B_L$. 
This suggests the interpretation of $B_L$ 
as an ``extended ramification locus'' of the map $f$;
for maps from $\PP^1_K$ to itself, this interpretation has been adopted
in the context of nonarchimedean dynamics
(e.g., see \cite{faber1, faber2}).
However, this picture is complicated by Remark~\ref{R:branch locus base extension}; roughly speaking, even for $x \in B_L$, the ``distance from $x$ to itself'' must be interpreted as a nonzero quantity.
In any case, Theorem~\ref{T:ramification} suggests the possibility of relating the full convergence polygon to more subtle measures of ramification, such as those considered recently by Temkin \cite{temkin1, temkin2}.
One older result in this direction is the theorem of Matsuda and Tsuzuki; see Theorem~\ref{T:Matsuda} below.
\end{remark}

\begin{theorem} \label{T:Matsuda}
Assume that $K$ has perfect residue field of characteristic $p>0$.
Suppose that $x \in \Gamma_{X,Z}$ is of type 2, choose $y \in f^{-1}(\{x\})$, and suppose that $\calH(y)$ is unramified over $\calH(x)$. Let $\kappa_x, \kappa_y$ be the residue fields of  the nonarchimedean fields $\calH(x), \calH(y)$. Let $\overline{\pi}$ be the restriction of $\pi$ along the identification of $H := \Gal(\kappa_y/\kappa_x) \cong \Gal(\calH(y)/\calH(x))$ with the stabilizer of $y$ in $G$.
\begin{enumerate}
\item[(a)]
The polygon $\calN(x)$ is zero.
\item[(b)]
Let $\vec{t}$ be a branch of $X$ at $x$, and let $v$ be the point on $C_x$ corresponding to
$\vec{t}$ (see Definition~\ref{D:branch}).
Then $\partial_{\vec{t}}(\calN)$ computes the wild Hasse-Arf polygon of $\overline{\pi}$ at $v$, i.e., the Newton polygon in which the slope $s \geq 0$ occurs with multiplicity $\dim_K(V^{H^{s+}}/V^{H^s})$. In particular, $\partial_{\vec{t}}(h(\calN))$ computes the Swan conductor of $\overline{\pi}$ at $v$.
\end{enumerate}
\end{theorem}
\begin{proof}
See \cite{tsuzuki-index}, \cite{matsuda}, or \cite{kedlaya-overview}.
\end{proof}

\begin{example} \label{exa:cyclic}
Set notation as in Example~\ref{exa:cyclic1}, except take $Z = \{0, 1, \infty\}$.
Let $x \in \Gamma_{X,Z}$ be the generic point of the disc $|z - 1| \leq \omega^p$,
which we can also write as $|u| \leq 1$ for $u = (z-1)/(\zeta_p - 1)^p$.
 Let $y$ be the unique preimage of $x$ in $Y$.
Let $t$ be the coordinate $(z-1)/(\zeta_p - 1)$ on $Y$; then
$\kappa_{x} = \FF_p(\overline{u})$
while
$\kappa_{y} = \FF_p(\overline{t})$
where $\overline{t}^p - \overline{t} = \overline{u}$.
For $\vec{t}$ the branch of $x$ towards $x_{1}$, we have $\partial_{\vec{t}}(h(\calN)) = 1$; as predicted by Theorem~\ref{T:Matsuda}, this equals the Swan conductor of the residual extension at $\infty$. 
\end{example}

\begin{remark} \label{R:cyclic change Z}
In Example~\ref{exa:cyclic}, it is necessary to include 1 in $Z$ to force $x$ into $\Gamma_{X,Z}$, so that Theorem~\ref{T:Matsuda} applies; otherwise, we would have
$\partial_{\vec{t}}(h(\calN)) = 0$ as in Example~\ref{exa:p-th root}. This provides an explicit example of the effect of enlarging $Z$ on the behavior of $\calN$, as described in Remark~\ref{R:enlarge Z for index}.
\end{remark}

\begin{remark}
To generalize Theorem~\ref{T:Matsuda} to cases where $\calH(y)$ is ramified over $\calH(x)$, it may be most convenient to use Huber's ramification theory for adic curves
\cite{huber}, possibly as refined in \cite{temkin1, temkin2}.
In a similar vein, the global index formula (Theorem~\ref{T:index formula}) is essentially the Riemann-Hurwitz formula for the map $f$, in which case it should be possible to match up the local contributions appearing in Theorem~\ref{T:index formula} with ramification-theoretic local contributions.
\end{remark}

\begin{remark}
Suppose that $p>0$, $X  = \PP^1_K$, 
and $f$ extends to a finite flat morphism of smooth curves
over $\frako_K$ with target $\PP^1_{\frako_K}$. By \cite[Theorem~3.1]{crew}, $\calE$ admits a unit-root Frobenius structure in a neighborhood of $x_1$,
from which it follows that $\calE$ satisfies the Robba condition at $x_1$.

Let $k$ be the residue field of $K$.
If in addition the points of $Z$ have distinct projections to $\PP^1_k$,
then by Remark~\ref{R:parameters}, $\calE$ satisfies the Robba condition everywhere.
By Remark~\ref{R:asympotic irregularity}, this implies that $\calE$ has exponents in $\ZZ_p$; by Remark~\ref{R:tame exponent calculation}, this implies that $f$ is tamely ramified (i.e., the inertia group of each point of $Y$ has order coprime to $p$). By contrast, if the points of $Z$ do not have distinct projections to $\PP^1_k$, then $f$ need not be tamely ramified; see 
\S\ref{sec:automorphisms disc}.
\end{remark}

\begin{remark} \label{R:finite morphism liouville}
For a connection derived from a finite morphism, in case $p>0$ the technical conditions of Remark~\ref{R:star2} are always satisfied. An important consequence of this statement for our present work is that the conclusion of Theorem~\ref{T:virtual local index}(c) can be established for such a connection even when $p>0$; see Remark~\ref{R:convexity positive characteristic}.
\end{remark}

\section{Artin--Hasse exponentials and Witt vectors}
\label{sec:Artin-Hasse}

We will conclude by specializing the previous discussion to cyclic covers of discs in connection with the Oort local lifting problem. 
In preparation for this, we need to recall some standard constructions of $p$-adic analysis.

\begin{hypothesis}
Throughout \S\ref{sec:Artin-Hasse}, fix a prime $p$ and a positive integer $n$.
In some algebraic closure of $\QQ_p$, fix a sequence of primitive $p^n$-th roots of unity $\zeta_{p^n}$ such that $\zeta_{p^n}^p = \zeta_{p^{n-1}}$.
\end{hypothesis}

\begin{defn}
The \emph{Artin-Hasse exponential series} at $p$ is the formal power series
\[
E_p(t) := \exp\left( \sum_{i=0}^\infty \frac{t^{p^i}}{p^i}
\right).
\]
\end{defn}

\begin{lemma} \label{L:Artin-Hasse}
We have $E_p(t) \in \ZZ_{(p)}\llbracket t \rrbracket$.
In particular, $E_p(t)$ converges for $|t| < 1$.
\end{lemma}
\begin{proof}
See for instance \cite[Proposition~9.9.2]{kedlaya-book}. 
\end{proof}

\begin{lemma} \label{L:value at 1}
Let $\ZZ_{(p)} \langle t \rangle$ be the subring of $\ZZ_{(p)} \llbracket t \rrbracket$ consisting of series $\sum_{i=0}^\infty c_i t^i$ for which the $c_i$ converge $p$-adically to $0$ as $i \to \infty$.
Then if we define the power series
\[
f(z,t) := \frac{E_p(z t) E_p(t^p)}{E_p(t) E_p(z t^p)} \in
 \ZZ_{(p)}\llbracket t,1-z \rrbracket
\]
using Lemma~\ref{L:Artin-Hasse},
we have
\[
f(z,t) \in z + (t-1) \ZZ_{(p)}\langle t \rangle \llbracket 1-z \rrbracket.
\]
\end{lemma}
\begin{proof}
Write $f(z,t) = \exp g(z,t)$ with
\begin{align*}
g(z,t) &= \sum_{i=0}^\infty \frac{(z^{p^i}-1)(t^{p^i} - t^{p^{i+1}})}{p^i}\\
&= \sum_{i=0}^\infty \sum_{j=1}^\infty (-1)^{j} \binom{p^i}{j} (1-z)^{j} \frac{t^{p^i} - t^{p^{i+1}}}{p^i} \\
&= \sum_{j=1}^\infty \frac{(-1)^{j}}{j}  (1-z)^{j} \sum_{i=0}^\infty \binom{p^i-1}{j-1} (t^{p^i} - t^{p^{i+1}}) \\
&= \sum_{j=1}^\infty \frac{(-1)^{j}}{j}  (1-z)^{j} \left( \binom{0}{j-1} t + \sum_{i=1}^\infty \left( \binom{p^i-1}{j-1}  - \binom{p^{i-1}-1}{j-1} \right) t^{p^i}
\right).
\end{align*}
For each fixed $j$, $\binom{p^i-1}{j-1}$ converges $p$-adically to $\binom{-1}{j-1} = (-1)^{j-1}$ as $i \to \infty$. It follows that
$g(z,t) \in (\ZZ_{(p)} \langle t \rangle \otimes_{\ZZ} \QQ) \llbracket 1-z \rrbracket$
and
\[
g(z,t) \equiv -\sum_{j=1}^\infty \frac{(1-z)^j}{j} \pmod{(t-1)(\ZZ_{(p)} \langle t \rangle \otimes_{\ZZ} \QQ)  \llbracket 1-z \rrbracket}.
\]
This then implies that $f(z,t) \in (\ZZ_{(p)} \langle t \rangle \otimes_{\ZZ} \QQ) \llbracket 1-z \rrbracket$ and 
\[
f(z,t) \equiv z \pmod{(1-t)(\ZZ_{(p)} \langle t \rangle \otimes_{\ZZ} \QQ) \llbracket 1-z \rrbracket}.
\]
Since $f(z,t)$ also belongs to $\ZZ_{(p)} \llbracket t, 1-z \rrbracket$
by Lemma~\ref{L:Artin-Hasse}, we may deduce the claimed inclusion.
\end{proof}

\begin{defn}
Let $W_n$ denote the $p$-typical Witt vector functor. Given a ring $R$, the set $W_n(R)$ consists of $n$-tuples $\underline{a} = (a_0,\dots,a_{n-1})$, and the arithmetic operations on $W_n(R)$ are characterized by functoriality in $R$ and the property that the \emph{ghost map} $w: W_n(R) \to R^n$
given by
\begin{equation} \label{eq:ghost map}
(a_0,\dots,a_{n-1}) \mapsto
(w_0,\dots,w_{n-1}), \qquad w_i = \sum_{j=0}^i p^j a_j^{p^{i-j}}
\end{equation}
is a ring homomorphism for the product ring structure on $R^n$.
For any ideal $I$ of $R$, let $W_n(I)$ denote the subset of $W_n(R)$
consisting of $n$-tuples with components in $I$; since $W_n(I) = \ker(W_n(R) \to W_n(R/I))$, it is an ideal of $W_n(R)$.

Various standard properties of Witt vectors may be derived by using functoriality to reduce to polynomial identities over $\ZZ$, then checking these over $\QQ$ using the fact that the ghost map is a bijection if $p^{-1} \in R$. Here are two key examples.
\begin{enumerate}
\item[(i)]
Define the \emph{Teichm\"uller map} sending $r \in R$ to $[r] := (r,0,\dots,0) \in W_n(R)$. Then this map is multiplicative: for all $r,s \in R$, $[rs] = [r] [s]$.
\item[(ii)]
Define the \emph{Verschiebung map} sending $\underline{a} \in W_n(R)$ to 
$V(\underline{a}) := (0, a_0,\dots,a_{n-2}) \in W_n(R)$. Then this map is additive: for all $\underline{a}, \underline{b} \in W_n(R)$, $V_n(\underline{a}+\underline{b}) = V_n(\underline{a}) + V_n(\underline{b})$.
\end{enumerate}
\end{defn}

\begin{defn}
In case $R$ is an $\FF_p$-algebra, the Frobenius endomorphism $\varphi: R \to R$ extends by functoriality to $W_n(R)$ and satisfies
\begin{equation} \label{eq:phi V char p}
p \underline{a} = (V \circ \varphi)(\underline{a}) = (\varphi \circ V)(\underline{a}) \qquad (\underline{a} \in W_n(R));
\end{equation}
see for instance \cite[\S 0.1]{illusie}. It follows that for general $R$, if we define the map $\sigma$ sending
$\underline{a} \in W_n(R)$ to $\sigma(\underline{a}) = (a_0^p, \dots, a_{n-1}^p) \in W_n(R)$,
then
\begin{equation} \label{eq:p-mult}
p \underline{a} - (V \circ \sigma)(\underline{a})  \in W_n(pR)
\qquad (\underline{a} \in W_n(R)).
\end{equation}
Beware that $\sigma$ is in general not a ring homomorphism.
\end{defn}

\begin{defn} \label{D:Artin-Hasse-Witt}
By Lemma~\ref{L:Artin-Hasse}, we have
\[
E_{n,p}(t) := \frac{E_p(\zeta_{p^n} t)}{E_p(t)} = \exp\left( \sum_{i=0}^{n-1} (\zeta_{p^{n-i}} - 1) \frac{ t^{p^i}}{p^i}
\right) \in \ZZ_{(p)}[\zeta_{p^n}]\llbracket t \rrbracket.
\]
We may also define
\[
E_{n,p}(\underline{a}) := \prod_{i=0}^{n-1} E_{n-i,p}(a_i) \in \ZZ_{(p)}[\zeta_{p^n}]\llbracket a_0,\dots,a_{n-1} \rrbracket,
\]
which we may also write as 
\begin{equation} \label{eq:ghost components}
E_{n,p}(\underline{a}) = \exp \left( \sum_{i=0}^{n-1} (\zeta_{p^{n-i}} - 1)\frac{w_i}{p^i} \right)
\end{equation}
for $w_i$ as in \eqref{eq:ghost map}. Consequently, 
in $\ZZ_{(p)}[\zeta_{p^n}]\llbracket a_0,\dots,a_{n-1},b_0,\dots,b_{n-1} \rrbracket$,
\begin{equation} \label{eq:Witt vector sum}
E_{n,p}(\underline{a}) E_{n,p}(\underline{b}) = E_{n,p}(\underline{a} + \underline{b}).
\end{equation}
\end{defn}

\begin{defn} \label{D:Artin-Hasse-Witt2}
By Lemma~\ref{L:Artin-Hasse}, we may define the formal power series
\[
F_{n,p}(t) := 
\frac{E_{n,p}(t)}{E_{n,p}(t^p)} = \exp \left( \sum_{i=0}^{n-1} \frac{(\zeta_{p^{n-i}}-1)(t^{p^{i}} - t^{p^{i+1}})}{p^i} \right) 
\in \ZZ_{(p)}[\zeta_{p^n}]\llbracket t \rrbracket
\]
and
\[
F_{n,p}(\underline{a}) := \frac{E_{n,p}(\underline{a})}{E_{n,p}(\sigma(\underline{a}))} = \prod_{i=0}^{n-1} F_{n-i,p}(a_i)  \in \ZZ_{(p)}[\zeta_{p^n}]\llbracket a_0,\dots,a_{n-1} \rrbracket.
\]
By \eqref{eq:ghost components}, 
in $\ZZ_{(p)}[\zeta_{p^n}]\llbracket a_0,\dots,a_{n-1},b_0,\dots,b_{n-1} \rrbracket$ we have
\begin{equation} \label{eq:Witt vector sum2}
F_{n,p}(\underline{a}) F_{n,p}(\underline{b}) = F_{n,p}(\underline{a} + \underline{b}) E_{n,p}(\sigma(\underline{a} + \underline{b}) - \sigma(\underline{a}) - \sigma(\underline{b}) ).
\end{equation}
We will see shortly (Lemma~\ref{L:overconvergent}) that
$F_{n,p}(t)$ has radius of convergence greater than 1, which implies an analogous assertion for $F_{n,p}(\underline{a})$. For $p>2$, 
this is shown in \cite[Proposition~1.10]{matsuda} using a detailed computational argument; our argument follows the more conceptual approach
given in \cite[Theorem~2.5]{pulita-rank1}.
\end{defn}

\begin{defn}
By Lemma~\ref{L:Artin-Hasse}, we may define the formal power series
\[
G_{n,p}(\underline{a}) = \frac{E_{n,p}(p\underline{a})}{E_{n-1,p}(\underline{a})} \in \ZZ_{(p)}[\zeta_{p^n}]\llbracket a_0,\dots,a_{n-1} \rrbracket.
\]
By \eqref{eq:Witt vector sum},
\begin{equation} \label{eq:Witt vector sum3}
G_{n,p}(\underline{a}) G_{n,p}(\underline{b}) = G_{n,p}(\underline{a} + \underline{b}).
\end{equation}
We also have
\begin{equation} \label{eq:efg}
G_{n,p}(\underline{a}) = E_{n,p}(p\underline{a} - (0,a_0^p,\dots,a_{n-2}^p)) F_{n-1,p}(\underline{a})^{-1}.
\end{equation}
\end{defn}

\begin{lemma} \label{L:overconvergent}
\begin{enumerate}
\item[(a)]
We have
\[
G_{n,p}(\underline{a}) \in \ZZ_{(p)}[\zeta_{p^n}]\llbracket (\zeta_{p^n}-1) a_0,
\dots, (\zeta_p-1)a_{n-1} \rrbracket.
\]
\item[(b)]
The power series $F_{n,p}(\underline{a})$ converges on a polydisc
with radius of convergence strictly greater than $1$.
\end{enumerate}
\end{lemma}
\begin{proof}
Write
\begin{align*}
G_{n,p}(\underline{a}) &= \prod_{j=0}^{n-1} \exp \left( \sum_{i=0}^{n-j-1} (p(\zeta_{p^{n-j-i}} - 1) - (\zeta_{p^{n-j-1-i}}-1)) \frac{a_j^{p^i}}{p^i} \right)  \\
&= \prod_{j=0}^{n-1} \exp \left( \sum_{i=0}^{n-j-1} (\zeta_{p^{n-j-i}} - 1) (p - 1 - \zeta_{p^{n-j-i}} -\cdots - \zeta_{p^{n-j-i}}^{p-1}) \frac{a_j^{p^i}}{p^i} \right) \\
&= \prod_{j=0}^{n-1} E_{n-j,p}\left(( p - 1 - [\zeta_{p^{n-j}}]- \cdots - [\zeta_{p^{n-j}}^{p-1}])
[a_j]\right).
\end{align*}
Note that $ p - 1 - [\zeta_{p^{n-j}}]- \cdots - [\zeta_{p^{n-j}}^{p-1}]$ maps to zero in $W_{n-j}(\FF_p)$ and so belongs to $W_{n-j}((\zeta_{p^{n-j}}-1)\ZZ_{(p)}[\zeta_{p^{n-j}}])$. This proves (a).

To prove (b), apply \eqref{eq:efg} to write
\[
F_{n,p}(\underline{a}) = G_{n+1,p}(\underline{a})^{-1} E_{n+1,p}(p \underline{a} - (V \circ \sigma)(\underline{a})).
\]
Since $G_{n+1,p}(0) = 1$, (a) implies that $G_{n+1,p}(\underline{a})^{-1}$ converges on a polydisc with radius of convergence strictly greater than 1.
Since $E_{n+1,p}(\underline{a}) \in \ZZ_{(p)}[\zeta_{p^{n+1}}] \llbracket a_0,\dots,a_{n-1} \rrbracket$
and \eqref{eq:p-mult} implies that $p \underline{a} - (V \circ \sigma)(\underline{a}) \in W_n(pR)$, $E_{n+1,p}(p \underline{a} - (V \circ \sigma)(\underline{a}))$ also converges on a polydisc with radius of convergence strictly greater than 1. This proves (b).
\end{proof}

\begin{lemma} \label{L:identify root}
For $m \in \ZZ$, we have
\begin{equation} \label{eq:root of unity}
F_{n,p}(\underline{a}+\underline{b}+m) F_{n,p}(\underline{a})^{-1} E_{n,p}(
\sigma(\underline{a} + \underline{b}+m) - \sigma(\underline{a}) - \sigma(\underline{b}) - m) = \zeta_{p^n}^m
\end{equation}
as an equality of elements of $\ZZ_p[\zeta_{p^n}]\llbracket a_0,\dots,a_{n-1},b_0,\dots,b_{n-1} \rrbracket$.
\end{lemma}
Note that the presence of $m$ prevents us from heedlessly applying
\eqref{eq:Witt vector sum}, because for instance $E_{n,p}(m)$ does not make sense. (We like to think of this as an example of \emph{conditional convergence} in a nonarchimedean setting.) Similarly, we must work over $\ZZ_p$ rather than $\ZZ_{(p)}$.
\begin{proof}
Convergence of $F_{n,p}(\underline{a} + \underline{b}+m)$ is guaranteed by Lemma~\ref{L:overconvergent}. Convergence of 
$E_{n,p}(
\sigma(\underline{a} + \underline{b}+m) - \sigma(\underline{a}) - \sigma(\underline{b}) - m)$ is guaranteed by Lemma~\ref{L:Artin-Hasse}
and the fact that 
\[
\sigma(\underline{a} + \underline{b}+m) - \sigma(\underline{a}) - \sigma(\underline{b}) - m \in W_n((\zeta_{p^n}-1) \ZZ_p[\zeta_{p^n}]\llbracket a_0,\dots,a_{n-1},b_0,\dots,b_{n-1} \rrbracket).
\]
Thus the left side of \eqref{eq:root of unity} is well-defined.
Using \eqref{eq:Witt vector sum}, we see that this quantity is constant as a power series in $a_0,\dots,a_{n-1},b_0,\dots,b_{n-1}$; it thus remains to prove that
\begin{equation} \label{eq:root of unity2}
F_{n,p}(m) E_{n,p}(\sigma(m) - m) = \zeta_{p^n}^m.
\end{equation}
Using \eqref{eq:Witt vector sum} and \eqref{eq:Witt vector sum2}, we see that for $m,m' \in \ZZ$,
\begin{align*}
&F_{n,p}(m) E_{n,p}(\sigma(m) - m) 
F_{n,p}(m') E_{n,p}(\sigma(m') - m') \\
&\qquad = F_{n,p}(m+m') E_{n,p}(\sigma(m+m') - \sigma(m) - \sigma(m'))
E_{n,p}(\sigma(m) - m) E_{n,p}(\sigma(m') - m')  \\
&\qquad = F_{n,p}(m+m') E_{n,p}(\sigma(m+m') - m-m'),
\end{align*}
so both sides of \eqref{eq:root of unity2} are multiplicative in $m$.
It thus suffices to check \eqref{eq:root of unity2} for $m=1 = (1,0,\dots)$, in which case $\sigma(m) = m$ and so the desired equality becomes $F_{n,p}(1) = \zeta_{p^n}$. This follows from
Lemma~\ref{L:value at 1} by evaluating at $z = \zeta_{p^n}$.
\end{proof}

\section{Kummer-Artin-Schreier-Witt theory}
\label{sec:kasw}

In further preparation for discussion of the Oort local lifting problem,
we describe a form of \emph{Kummer-Artin-Schreier-Witt theory} for 
cyclic Galois extensions of a power series field.

\begin{hypothesis} \label{H:kasw}
Throughout \S\ref{sec:kasw},
fix a positive integer $n$,
assume that $K$ is discretely valued,
the residue field $k$ of $K$ is algebraically closed of characteristic $p>0$,
and $K$ contains a primitive $p^n$-th root of unity $\zeta_{p^n}$.
Put $F = k((\overline{z}))$.
\end{hypothesis}

\begin{defn}
For $\rho \in (0,1)$, let $A(\rho,1)$ be the annulus $\rho < |z| < 1$ in $\PP^1_K$;
the analytic functions on $A(\rho,1)$ can be viewed as certain Laurent series in $z$.
The union of the rings $\calO(A(\rho,1))$ over all $\rho \in (0,1)$ 
is called the \emph{Robba ring} over $K$ and will be denoted $\calR$.
(This ring can be interpreted as the local ring of the adic point of $\PP^{1,\an}_K$ specializing $x_1$ in the direction towards $0$.)

Let $\calR^{\bd}$ be the subring of $\calR$ consisting of formal sums with bounded coefficients; these are exactly the elements of $\calR$ which define \emph{bounded} analytic functions on $A(\rho,1)$ for some $\rho \in (0,1)$.
The ring $\calR^{\bd}$ carries a multiplicative \emph{Gauss norm} defined by
\[
\left| \sum_{i \in \ZZ} a_i z^i \right| = \max_i \{\left|a_i\right|\};
\]
let $\calR^{\inte}$ be the subring of $\calR^{\bd}$ consisting of elements of Gauss norm at most 1.
\end{defn}

\begin{lemma} \label{L:henselian}
The ring $\calR^{\inte}$ is a henselian discrete valuation ring.
Consequently, $\calR^{\bd}$ is a henselian local field with residue field $F$.
\end{lemma}
\begin{proof}
See \cite[Proposition~3.2]{matsuda1}.
\end{proof}

We next prepare to formulate the comparison between Kummer theory and Artin-Schreier-Witt theory by introducing the two sides of the comparison.
\begin{defn} \label{D:Kummer}
For any field $L$ of characteristic not equal to $p$, for $L^{\sep}$ a separable closure of $L$, taking 
Galois cohomology on the exact sequence
\[
1 \to \mu_{p^n} \to (L^{\sep})^{\times} \stackrel{\bullet^{p^n}}{\to} (L^{\sep})^\times \to 1
\]
of $G_L$-modules gives the \emph{Kummer isomorphism}
\[
L^{\times}/L^{\times p^n} \cong H^1(G_L, \mu_{p^n})
\]
because $H^1(G_L, (L^{\sep})^{\times}) = 0$ by Noether's form of Hilbert's Theorem 90.

In the case $L = \calR^{\bd}$, by Lemma~\ref{L:henselian}
we have a surjection $G_L \to G_F$ identifying $G_F$ with the quotient of the maximal unramified extension of $L$. We thus obtain a restriction map $H^1(G_F, \mu_{p^n}) \to H^1(G_L, \mu_{p^n})$
and thus a map $H^1(G_F, \mu_{p^n}) \to L^{\times}/L^{\times p^n}$.
Note that $G_F$ acts trivially on $\mu_{p^n}$, so we may identify
$\mu_{p^n}$ as a $G_F$-module with $\ZZ/p^n \ZZ$ by identifying our chosen primitive $p^n$-th root of unity $\zeta_{p^n} \in \mu_{p^n}$
with $1 \in \ZZ/p^n \ZZ$. We thus end up with a homomorphism
\begin{equation} \label{eq:Kummer map}
H^1(G_F, \ZZ/p^n \ZZ) \to (\calR^{\bd})^\times / (\calR^{\bd})^{\times p^n}.
\end{equation}
(Beware that the opposite sign is used to normalize the isomorphism $\mu_{p^n} \cong \ZZ/p^n \ZZ$ in \cite{matsuda}, \cite{pulita-rank1}, leading to some minor differences in the formulas.)
\end{defn}

\begin{defn} \label{D:Artin-Schreier-Witt}
Consider the exact sequence
\[
0 \to \ZZ/p^n \ZZ = W_n(\FF_p) \to W_n(F^{\sep}) \stackrel{1-\varphi}{\to} W_n(F^{\sep}) \to 0
\]
where $\varphi$ denotes the Frobenius endomorphism of $W_n(F^{\sep})$.
The additive group $W_n(F^{\sep})$ is a successive extension of copies of the additive group of 
$F^{\sep}$; since $H^1(G_{F}, F^{\sep}) = 0$ by the additive version of Theorem 90, 
we also have $H^1(G_{F}, W_n(F^{\sep})) = 0$. We thus obtain the \emph{Artin-Schreier-Witt isomorphism}
\begin{equation} \label{eq:asw}
\coker(1-\varphi, W_n(F)) \cong H^1(G_{F}, \ZZ/p^n \ZZ).
\end{equation}
Combining this isomorphism with the map \eqref{eq:Kummer map} derived from the Kummer isomorphism, we obtain a homomorphism
\begin{equation} \label{eq:kasw}
\coker(1-\varphi, W_n(F)) \to (\calR^{\bd})^\times / (\calR^{\bd})^{\times p^n}.
\end{equation}
\end{defn}

We note in passing how Swan conductors appear in the Artin-Schreier-Witt isomorphism.
See also \cite[Theorem~3.2]{kato-swan}, \cite[Theorem~1.1]{garuti2},
\cite[Proposition~4.2]{thomas}; see especially
\cite[Theorem 4.9]{chiarellotto-pulita} for the full computation of $\coker(1-\varphi, W_n(F))$.
\begin{lemma} \label{L:identify conductor}
For $\overline{\underline{a}} \in W_n(F)$, 
let $\pi: G_{F} \to K^\times$ be the character corresponding via \eqref{eq:asw}
to the class of $\overline{\underline{a}}$  in $\coker(\varphi-1, W_n(F))$.
For $j=0,\dots,n-1$, let $m_j$ be the negation of the $\overline{z}$-adic valuation of $\overline{a}_j$, and assume that $m_j$ is not a positive multiple of $p$.
Then the Swan conductor of $\pi$ equals 
\[
\max\{0, m_0 p^{n-1},m_1 p^{n-2}, \dots, m_{n-1}\}.
\]
\end{lemma}
\begin{proof}
By hypothesis, if $m_j$ is positive then it is not divisible by $p$,
so $m_j p^{n-1-i}$ has $p$-adic valuation $n-i-1$.
Consequently, any two of the quantities $m_0 p^{n-1},m_1 p^{n-2}, \dots, m_{n-1}$, if they are nonzero, must be distinct. 
It thus suffices to check the claim in case $\overline{\underline{a}}$
is a Teichm\"uller element $[\overline{a}]$ for some $\overline{a} \in F$
of $\overline{z}$-adic valuation $-m$ for some integer $m$ which is positive and not divisible by $p$. By splitting $\overline{a}$ into powers of $\overline{z}$,
we may further reduce to the case $\overline{a} = c \overline{z}^{-m}$.
Using the compatibility of Swan conductors with tame base extensions, we may further reduce to the case $\overline{a} = \overline{z}^{-1}$.

For $j=1,\dots,n$, let $F_{j}$ be the extension of $F$ obtained by adjoining 
the coordinates $\overline{b}_0, \dots,\overline{b}_{j-1}$ of a Witt vector $\overline{\underline{b}}$ satisfying
\[
\overline{\underline{b}} - \varphi(\overline{\underline{b}}) = [\overline{a}].
\]
It is clear that $\overline{b}_0 - \overline{b}_0^p = \overline{a}$.
By direct computation, one sees that for $j=1,\dots,n-1$, $\overline{b}_j - \overline{b}_j^p$ has the same $\overline{z}$-adic valuation as $\overline{b}_{j-1} \overline{a}^{p^j-p^{j-1}}$ and that this valuation is $-(p^j - p^{j-1} + \cdots + p^{-j})$. It follows that the breaks in the lower numbering filtration of $\Gal(F_n/F)$ occur at 
$(p^{2j+1}+1)/(p+1)$ for $j=0,\dots,n-1$;  by Herbrand's formula
\cite[Chapter~IV]{serre}, the breaks in the upper numbering filtration occur at $1, p, \dots, p^{n-1}$. The last of these breaks is the Swan conductor of $\pi$, proving the claim.
\end{proof}
\begin{cor}
With notation as in Lemma~\ref{L:identify conductor}, 
if $b_j$ is the Swan conductor of $\pi^{\otimes p^{n-j}}$, then $b_j \geq p b_{j-1}$ for $j=1,\dots,n$.
\end{cor}
\begin{proof}
The representation $\pi^{\otimes p^{n-j}}$ corresponds via \eqref{eq:asw}
to the class of $p^{n-j} \overline{\underline{a}}$ in $\coker(\varphi-1,W_j(F))$.
By \eqref{eq:phi V char p}, this class is also represented by $V^{n-j}(\overline{\underline{a}})$. We may now apply Lemma~\ref{L:identify conductor} to deduce the claim.
\end{proof}

We now make explicit the relationship between the Kummer and Artin-Schreier-Witt isomorphisms.
\begin{theorem}[Matsuda] \label{T:matsuda classification}
The homomorphism \eqref{eq:kasw} is induced by a homomorphism
\begin{equation} \label{eq:matsuda hom}
W_n(\calR^{\inte}) \to (\calR^{\inte})^\times,
\qquad
\underline{a} \mapsto E_{n,p}(p^n \underline{a}).
\end{equation}
\end{theorem}
\begin{proof}
We first clarify the interpretation of the expression $E_{n,p}(p^n \underline{a})$ as an element of $\calR^{\inte}$. Since by definition $E_{n,p}(\underline{a}) = \prod_{i=0}^{n-1} E_{n-i,p}(a_i)$, this amounts to evaluating
the formal power series $E_{n,p}(p^n t)$ at $t=g$ for some $g = \sum_{i \in \ZZ} g_i z^i \in \calR^{\inte}$.
By Lemma~\ref{L:Artin-Hasse}, there exists $\rho_1 > 1$ such that the series $E_{n,p}(p^n t)$ converges for $|t| < \rho_1$.  By the definition of $\calR$, there exists $\rho_2 \in (0,1)$ such that for any $z$ in any nonarchimedean field containing $K$ with $\rho_2 < |z| < 1$, $|g_i z^i| < \rho_1$ for all but finitely many $i<0$. The same remains true if we replace $\rho_2$ by any larger value in $(0,1)$; by so doing, we can force the inequality $|g_i z^i| < \rho_1$ to hold for all $i<0$. The same inequality also holds for $i \geq 0$ because $|g_i| \leq 1$ by the definition of $\calR^{\inte}$; consequently, $g(z)$ belongs to the region of convergence of $E_{n,p}(p^n t)$, and the evaluation $E_{n,p}(p^n g)$ is well-defined as an analytic function on the annulus $A(\rho_2,1)$. Since both $E_{n,p}(p^n t)$ and $g$ have coefficients in $\frako_K$, the same is true of the composition, so $E_{n,p}(p^n g) \in \calR^{\inte}$.

By \eqref{eq:Witt vector sum}, the map \eqref{eq:matsuda hom} is a homomorphism.
Given $\underline{a}$, choose a minimal finite separable 
extension $S$ of $F$  
such that there exists $\overline{\underline{b}} \in W_n(S)$ with
\[
\overline{\underline{b}} -
\varphi(\overline{\underline{b}}) = \overline{\underline{a}}.
\]
(This amounts to forming a tower of Artin-Schreier extensions over $F$.)
Apply Lemma~\ref{L:henselian} to construct a finite \'etale algebra
$\calS^{\inte}$ over $\calR^{\inte}$ with residue field $S$. Choose a lift $\underline{b} \in \calS^{\inte}$ of $\overline{\underline{b}}$; 
then $\underline{a} + \sigma(\underline{b}) - \underline{b} \in W_n(p \calS^{\inte})$.
By Lemma~\ref{L:overconvergent}(b), we may define an element
\[
f := F_{n,p}(\underline{b}) E_{n,p}( \underline{a} + \sigma(\underline{b}) - \underline{b} ) \in \calS^{\inte}.
\]
Then $f^{p^n} = E_{n,p}(p^n \underline{a})$.

On one hand, the image of $\underline{a}$ in $\coker(\varphi-1, W_n(F))$
corresponds to the element of $H^1(G_F, \ZZ/p^n \ZZ)$ which factors through $H^1(\Gal(S/F), \ZZ/p^n \ZZ)$ and sends $g \in \Gal(S/F)$ to the integer $m \in \ZZ/p^n \ZZ$ for which $\varphi(\overline{\underline{b}}) = \overline{\underline{b}} + m$. On the other hand,
we have $g(\underline{b}) = \underline{b} + m + \underline{c}$
for some $\underline{c} \in W((\zeta_{p^n}-1) \calR_n^{\inte})$, and so
\begin{align*}
g(f) &= F_{n,p}(\underline{b} + \underline{c} + m) E_{n,p}( \underline{a} + \sigma(\underline{b} + \underline{c} + m) - \underline{b}  - \underline{c} - m) \\
&= F_{n,p}(\underline{b} + \underline{c} + m) F_{n,p}(\underline{b})^{-1} E_{n,p}(\sigma(\underline{b} + \underline{c} + m) -\sigma(\underline{b}) - \underline{c} - m)
F_{n,p}(\underline{b}) E_{n,p}( \underline{a} + \sigma(\underline{b}) - \underline{b}) \\
&= \zeta_{p^n}^m f
\end{align*}
by Lemma~\ref{L:identify root}.
It follows that \eqref{eq:matsuda hom} induces \eqref{eq:kasw} as desired.
\end{proof}

\begin{remark}
By comparing a character with its $p$-th power, we may deduce from
Theorem~\ref{T:matsuda classification} that
\[
\frac{E_{n,p}(p^n \underline{a})}{E_{n-1,p}(p^{n-1} \underline{a})}
\in (\calR^{\inte})^{\times p^{n-1}}.
\]
This may also be seen directly from
Lemma~\ref{L:overconvergent} by rewriting the left side as
$G_{n-1,p}(\underline{a})^{p^{n-1}}$. 
\end{remark}

\section{Automorphisms of a formal disc}
\label{sec:automorphisms disc}

To conclude, we use Kummer-Artin-Schreier-Witt theory to translate the Oort local lifting problem into a question about the construction of suitable connections on $\PP^1_K$,
and use this interpretation to describe existing combinatorial invariants connected with the Oort problem in terms of convergence polygons.

\begin{hypothesis}
Throughout \S\ref{sec:automorphisms disc}, retain Hypothesis~\ref{H:kasw}
and additionally fix $\underline{a} \in W_n(\calR^{\inte})$.
\end{hypothesis}

\begin{defn} \label{D:finite cover}
Let $\pi: G_F \to \mu_{p^n}$ be the character corresponding to $\underline{a}$ via the maps
$W_n(\calR^{\inte}) \to \coker(\varphi-1, W_n(F)) \cong H^1(G_F, \mu_{p^n})$ (the latter isomorphism being \eqref{eq:asw}).
For $i=1,\dots,n$, let $b_i$ be the Swan conductor of $\pi^{\otimes p^{n-i}}$.

As before, let $x_1 \in \PP^{1,\an}_K$ denote the generic point of the disc $|z| < 1$.
The residue field $\calH(x_1)$ is the fraction field of a Cohen ring for the field $k(\overline{z})$.
We have an embedding of $\calH(x_1)$ into the completion of $\calR^{\bd}$ for the Gauss norm, arising from an inclusion of Cohen rings lifting the inclusion $k(\overline{z}) \subset F$.

Apply the Katz-Gabber construction \cite{katz-extension} to lift $\pi$ to a representation of $G_{k(\overline{z})}$ unramified away from $\{0,\infty\}$, then identify the latter with a representation $\tilde{\pi}: G_{\calH(x_1)} \to \mu_{p^n}$.
By Crew's analogue of the Katz-Gabber construction for $p$-adic differential equations \cite{crew-canonical},
the character $\tilde{\pi}$ arises from a finite Galois cover of $A(\rho_1, \rho_2)$
for some $\rho_1 < 1 < \rho_2$.
We may then proceed as in Definition~\ref{D:rep to connection} to obtain a rank 1 bundle
$\calE_{n}$ with connection on this subspace. As in \cite[Theorem~3.1]{pulita-rank1}
(modulo a sign convention; see Definition~\ref{D:Kummer}), this connection
can be described explicitly as the free vector bundle on a single generator $\bv$ equipped with the connection 
\[
\nabla(\bv) = \sum_{i=0}^{n-1} \sum_{j=0}^{n-i-1} (\zeta_{p^{n-j}}-1) a_i^{p^j-1}  \bv \otimes da_i.
\]
Formally, we have $d\bv = \bv \otimes d\log E_{n,p}(\underline{a})$.

For $i=1,\dots,n$, put $\calE_i = \calE_n^{\otimes p^{n-i}}$. Then $\calE_i$ corresponds to the character $\pi^{\otimes p^{n-i}}$ of order $p^i$ in a similar fashion.
\end{defn}

\begin{remark} \label{R:finite cover uniqueness}
The construction given in Definition~\ref{D:finite cover}
is precisely the $(\varphi, \nabla)$-module associated to $\pi$ by the work of Fontaine and Tsuzuki \cite{tsuzuki-monodromy}. In particular, it is unique in the sense that given any other construction, the two become isomorphic on $A(\rho_1, \rho_2)$ for some convenient choice of $\rho_1 < 1 < \rho_2$. Similarly, the restriction of the connection
to $A(1, \rho_2)$ is the $(\varphi, \nabla)$-module associated to the restriction of $\pi$ to the decomposition group at $\infty$, and enjoys a similar uniqueness property.

In fact, we can say something stronger. Recall that the Katz extension of a representation of $G_F$ has only tame ramification at $\infty$. Since $\pi$ factors through a $p$-group, its Katz extension must in fact be unramified at $\infty$, so it descends to a representation of the \'etale fundamental group of $\PP^1_K - \{0\}$. Such a representation gives rise to an overconvergent $F$-isocrystal, as shown by Crew
\cite[Theorem~3.1]{crew-isocrystals}; this means that for some choice of $\rho_1 < 1 < \rho_2$, the connection on $A(\rho_1, \rho_2)$ described above extends to the subspace
$|z| > \rho_1$ of $\PP^1_K$.
\end{remark}

\begin{defn}
Let $S$ be the fixed field of $\ker(\pi)$; we may identify $S$ with $k((\overline{u}))$
for some parameter $\overline{u}$. The action of $G_F$ defines a continuous $k$-linear 
action $\tau$ of $\mu_{p^n}$ on $k \llbracket \overline{u} \rrbracket$. A \emph{solution of the lifting problem} for $\pi$ is a lifting of $\tau$ to a continuous $\frako_K$-linear action $\tilde{\tau}$ of $\mu_{p^n}$ on $\frako_K \llbracket u \rrbracket$.
\end{defn}

\begin{conj}[Oort] \label{conj:Oort}
A solution of the lifting problem exists for every $\pi$.
\end{conj}

A spectacular breakthrough on this problem has been made recently in work of Obus--Wewers and Pop \cite{obus-wewers, pop}.
\begin{theorem} \label{T:Oort}
For fixed $\pi$, a solution of the lifting problem exists over some finite extension of $K$
(that is, the lifting problem is solved if we do not insist on the field of definition).
\end{theorem}

We will not say anything more here about the techniques used to prove Theorem~\ref{T:Oort}. Instead, we describe an equivalence between solutions of the lifting problem for $\pi$ and extensions of the connection on $\calE_{n}$.

\begin{defn} \label{D:solution of lifting problem}
Suppose that $\tilde{\tau}$ is a solution of the lifting problem. 
Then $\tilde{\tau}$ gives rise to a finite Galois cover of the disc $|z| < 1$, which 
thanks to Remark~\ref{R:finite cover uniqueness}
may be glued together with the cover from Definition~\ref{D:finite cover} to give a finite ramified cover $f_n: Y_n \to X$ with $X = \PP^1_K$, such that $x_1$ has a unique preimage in $Y_n$. 
(This cover is constructed \emph{a priori} at the level of analytic spaces, but descends to a cover of schemes by rigid GAGA; see Remark~\ref{R:index conditions}.)
For $i=1,\dots,n$, let $f_i: Y_i \to X$ be the cover corresponding to $\rho^{p^{n-i}}$ in similar fashion, and let $Z_i$ be the ramification locus of $f_i$;
also put $Z_0 = \emptyset$. For each $x \in Z_i - Z_{i-1}$, $\calE_n$ is regular at $x$ with exponent $m/p^{n-i+1}$ for some $m \in \ZZ - p\ZZ$.
\end{defn}

Using the Riemann-Hurwitz formulas in characteristic $0$ and $p$, we obtain the following relationship between the ramification of $\tau$ and of $f_n$.

\begin{lemma} \label{L:RH comparison}
With notation as in Definition~\ref{D:solution of lifting problem}, for $i=1,\dots,n$,
\begin{align}
\label{eq:rh1}
2-2g(Y_i) &= 2p^i - \sum_{j=1}^i (p^i - p^{j-1}) (\length(Z_j) - \length(Z_{j-1})) \\
\label{eq:rh2}
&=2p^i - \sum_{j=1}^i (p^j - p^{j-1}) b_j.
\end{align}
Consequently,
\begin{equation} \label{eq:rh3}
\length(Z_i) = b_i + 1 \qquad (i=1,\dots,n).
\end{equation}
\end{lemma}
\begin{proof}
The Rieman-Hurwitz formula for $f_i$ asserts that
\[
2 - 2g(Y_i) = \deg(f_i)(2-2g(\PP^1_K)) - \sum_{z \in X} (\deg(f_i) - \length(f_i^{-1}(z))).
\]
For $z \in X$, we have $\deg(f_i) - \length(f_i^{-1}(z)) = 0$ unless $z \in Z_i$.
If $z \in Z_j - Z_{j-1}$ for some $j \in \{1,\dots,i\}$, then 
each preimage of $z$ in $Y_i$ is fixed by a group of order $p^{i-j+1}$, so 
$\length(f_i^{-1}(z)) = p^{j-1}$.
This yields \eqref{eq:rh1}.

Let $\overline{f}_i: \overline{Y}_i \to \PP^1_k$ be the reduction of $f_i$;
then $g(\overline{Y}_i) = g(Y_i)$.
Set notation as in the proof of Lemma~\ref{L:identify conductor}.
Since $\overline{f}_i$ is Galois and only ramifies above $0$,
the Riemann-Hurwitz formula for $\overline{f}_i$ (see \cite[Corollary~3.4.14, Theorem~3.8.7]{stichtenoth}) can be written in the form
\[
2 - 2g(\overline{Y}_i) = \deg(f_i)(2-2g(\PP^1_k)) - \sum_{m=0}^\infty (\#\Gal(F_i/F)_m - 1)
\]
where $\Gal(F_i/F)_m$ denotes the $m$-th subgroup of $\Gal(F_i/F)$ in the lower numbering filtration. By identifying these subgroups as in the proof of Lemma~\ref{L:identify conductor}, we obtain \eqref{eq:rh2}. By combining \eqref{eq:rh1} and \eqref{eq:rh2}, we may solve for 
$b_i$ for $i=1,\dots,n$ in succession to obtain \eqref{eq:rh3}.
\end{proof}

We have the following explicit version of Theorem~\ref{T:continuous2} in this setting.
\begin{theorem} \label{T:continuous for lifting problem}
Retain notation as in Definition~\ref{D:solution of lifting problem}.
Let $\Gamma = \Gamma_{X, Z_n \cup \{\infty\}}$ be the union of the paths from $\infty$ to the elements of $Z_n$.
Let $\calN_n$ be the the convergence polygon of $\calE_n$.
\begin{enumerate}
\item[(a)]
The function
$\calN_n$ factors through the retraction of $X^{\an}$ onto $\Gamma$,
and is affine on each edge of $\Gamma$. 
\item[(b)]
Let $\Gamma_\infty \subset \Gamma$ be the path from $x_1$ to $\infty$.
For each $x \in \Gamma_\infty$,  
$s_1(\calN_n(x)) = 0$.
\item[(c)]
The measure $\chi(\calE)$ (more precisely, the measure $\chi(\calE, Z_n)$ in the notation of Remark~\ref{R:enlarge Z for index}) is discrete, supported at $x_1$, of total measure $1-b_n$.
\item[(d)]
For each $x \in \Gamma - \Gamma_\infty$, for $\vec{t}$ the branch of $x$ towards $x_1$, $\partial_{\vec{t}} s_1(\calN_n) = \ell-1$ where $\ell$ is the length of the subset of $Z_n$ dominated by $x$.
\item[(e)]
For $i=1,\dots,n$, for each $x \in Z_i - Z_{i-1}$, let $\Gamma_x$ be the pendant edge of $\Gamma$ terminating at $x$. For each $y \in \Gamma_x$,
\begin{equation} \label{eq:lifting problem constant value}
s_1(\calN_n(y)) = \left( n-i+1 + \frac{1}{p-1} \right) \log p.
\end{equation}
\end{enumerate}
\end{theorem}
\begin{proof}
By Theorem~\ref{T:Matsuda}(a), we have $s_1(\calN_n(x_1)) = 0$; by Theorem~\ref{T:transfer}(a), $\calN_n$ is identically zero on all $x$ outside of the closed unit disc. In particular, we deduce (b).

By Theorem~\ref{T:virtual local index}(c) and Remark~\ref{R:finite morphism liouville}, we must have $\chi_x(\calE_n) \leq 0$ for all $x \in X^{\an}$.
By Theorem~\ref{T:index formula} (as reformulated in Theorem~\ref{T:local index formula}),
Lemma~\ref{L:index formula from irregularity},  and Lemma~\ref{L:RH comparison}, we have
\[
\int_{X^{\an}} \chi(\calE_n) = 
\chi_{\dR}(X^{\an}, \calE_n) = \chi_{\dR}(U^{\an}) = 2 - \length(Z_n) = 1 - b_n.
\] 
By the previous paragraph plus Theorem~\ref{T:Matsuda}(b), $\chi_{x_1}(\calE) = 1 -b_n$;
we thus deduce (c), which in turn immediately implies (a). Using Theorem~\ref{T:local index formula} (which again applies in this situation thanks to Remark~\ref{R:finite morphism liouville}), we deduce (d).
To obtain (e), in light of (a) we need only check \eqref{eq:lifting problem constant value} for $y$ in some neighborhood of $x$. This follows by noting that as in Example~\ref{exa:p-th root}, the binomial series $(1 + z)^{1/p^n}$ has radius of convergence $p^{-n-1/(p-1)}$.
\end{proof}

\begin{remark}
Retain notation as in Definition~\ref{D:solution of lifting problem}
and Theorem~\ref{T:continuous for lifting problem}. The union of the paths from $x_1$ to the points of $Z$ then form a tree on which $s_1(\calN_n)$ restricts to a harmonic function which is affine on each edge of the tree, with prescribed values and slopes at the pendant vertices (i.e., $x_1$ and the points of $Z$). However, the existence of such a function imposes strong combinatorial constraints on the relative positions of the points of $Z$; the resulting data are well-known in the literature on the Oort lifting problem, under the rubric of \emph{Hurwitz trees} \cite{brewis-wewers}. Similar data arise in \cite{temkin1, temkin2}.
\end{remark}

\begin{remark}
Conversely, suppose $X = \PP^1_K$; $Z$ equals $\{\infty\}$ plus a subset of the open unit disc;
$\calE$ is a rank 1 vector bundle with connection on $U = X-Z$;
for each $z \in Z$, $\nabla$ is regular at $z$ with exponent in $p^{-n} \ZZ$;
and for some $\rho \in (0,1)$, the restriction of $\calE$ to the space $|z| > \rho$
is isomorphic to $\calE_n$.
The connection $\calE^{\otimes p^n}$ has only removable singularities, so it can be shown to be trivial either by passing to $\CC$ as in the proof of Lemma~\ref{L:index formula from irregularity} and invoking complex GAGA (using the Riemann-Hilbert correspondence
and the fact that $\PP^{1,\an}_{\CC}$ is simply connected), or by applying the Dwork transfer theorem (Theorem~\ref{T:transfer}(a)) to the disc $|z| < \sigma$ for some $\sigma > \rho$.

From this, we may see that $\calE$ arises as $\calE_n$ for some data as in Definition~\ref{D:finite cover}, i.e., that $\calE$ arises from a solution of the lifting problem. For example, to construct the finite map $f: Y \to X$, we may work analytically over $\CC$  (again as in the proof of Lemma~\ref{L:index formula from irregularity}), using complex GAGA to descend to the category of schemes. In the complex analytic situation, we may locally choose a generator $\bv$ of $\calE$, then form $Y$ from $X$ by adjoining $s^{1/p}$ where $s \bv^{\otimes p}$ is a nonzero horizontal section of  $\calE^{\otimes p^n}$. By similar considerations, we see that $f$ is Galois with group $\mu_{p^n}$ and that $\calE \cong \calE_n$.
\end{remark}

\begin{remark}
For a given $\pi$,
it should be possible to construct a moduli space of solutions of the lifting problem
in the category of rigid analytic spaces over $K$. Theorem~\ref{T:Oort} would then imply that this space is nonempty. Given this fact, it may be possible to derive additional results on the lifting problem, e.g., to resolve the case of dihedral groups. For $p>2$, this amounts to showing that if $\tau$ anticommutes with the involution $\overline{z} \mapsto -\overline{z}$, then the action of the involution $z \mapsto -z$ fixes some point of the moduli space.
\end{remark}

\appendix

\section{Convexity}
\label{sec:convexity}

In this appendix, we give some additional technical arguments 
needed for the proofs of Theorem~\ref{T:monotonicity} and
Theorem~\ref{T:virtual local index}(c), which do not appear elsewhere in the literature.
These arguments are not written in the same expository style as the rest of the main text;  for instance, they assume much more familiarity with the author's book \cite{kedlaya-book}.

\begin{defn}
For $I$ a subinterval of $[0, +\infty)$, let $R_I$ be the ring of rigid analytic functions on the space $|t| \in I$ within the affine $t$-line over $K$, as in \cite[Definition~3.1.1]{kedlaya-radii}.
\end{defn}

\begin{hypothesis} \label{H:convexity}
Throughout Appendix~\ref{sec:convexity},
assume $p=0$,
and let $(M,D)$ be a differential module of rank $n$ over $R_{[0,\beta)}$. 
For $I$ a subinterval of $[0,\beta)$, write $M_I$ as shorthand for $M \otimes_{R_{[0,\beta)}} R_I$.
\end{hypothesis}

\begin{defn}
Let $\DD_\beta$ be the Berkovich disc $|t| < \beta$ over $K$.
For $x \in \DD_\beta$, define the real numbers $s_i(M, x)$ for $i=1,\dots,n$
as in \cite[Definition~4.3.2]{kedlaya-radii}; note that they are invariant under extension of $K$ \cite[Lemma~4.3.3]{kedlaya-radii}.
For $r > -\log \beta$, define the functions
\begin{align*}
g_i(M,r) &= -\log s_i(M, x_{e^{-r}}) \\
G_i(M,r) &= g_1(M,r) + \cdots + g_i(M,r).
\end{align*}
\end{defn}

We begin with a variant of Theorem~\ref{T:virtual local index}.
\begin{lemma} \label{L:H1 slope}
The right slope of $G_n(M,r)$ at $r = -\log \beta$ equals $-\dim_K H^1(M)$, provided that at least one of the two is finite.
\end{lemma}
\begin{proof}
Apply \cite[Theorem~3.5.2]{pulita-poineau5}.
\end{proof}

\begin{lemma} \label{L:monotonic2a}
For $\rho \in (0,\beta)$, let $x_\rho$ be the generic point of the disc $|t| \leq \rho$.
Then for $i=1,\dots,n$, the function $\rho \mapsto s_1(M,x_\rho) \cdots s_i(M, x_\rho)$ is nonincreasing in $\rho$.
\end{lemma}
\begin{proof}
The case $i=n$ is immediate from Lemma~\ref{L:H1 slope}.
To deduce the case $i<n$, it suffices to work locally around some
$\rho_0 = e^{-r_0}$, and to check only those values of $i$ for which $g_i(M,r_0) > g_{i+1}(M,r_0)$.
For $\lambda \in K$,
let $M_\lambda$ be the differential module obtained from $M$ by adding $\lambda$ to $D$;
we may also view $M_\lambda$ as the tensor product of $M$ with the rank one differential module $N_\lambda$ on a single generator $\bv$ satisfying $D(\bv) = \lambda \bv$. From the tensor product description, for $j=1,\dots,n$ we have
\begin{equation} \label{eq:tensor product lambda}
g_j(M_\lambda,r) \leq \max\{g_j(M,r), g_1(N_\lambda,r)\}, \quad
g_j(M,r) \leq \max\{g_j(M_\lambda,r), g_1(N_\lambda^\dual,r)\}.
\end{equation}
In addition, by Example~\ref{exa:exponential}, we have 
\[
g_1(N_\lambda, r) = g_1(N_\lambda^\dual, r) = \max\left\{- \log \left| \beta \right|, \frac{1}{p-1} \log p + \log \left| \lambda \right| \right\}
\]
independently of $r$.

Choose an open subinterval $I$ of $(g_{i+1}(M, r_0), g_i(M, r_0))$.
After replacing $K$ with a finite extension (which does not affect $G_j(M,r)$),
we may choose $\lambda$ in such a way that $g_1(N_\lambda, r) = g_1(N_\lambda^\dual, r)$ is equal to a constant value contained in $I$. 
We then have $g_i(M, r) > g_1(N_\lambda,r) > g_{i+1}(M,r)$ for all $r$ in some neighborhood of $r_0$. For such $r$, for $j \leq i$ we apply \eqref{eq:tensor product lambda} to see that $g_j(M_\lambda,r) = g_j(M, r)$.

For $j > i$, \eqref{eq:tensor product lambda} implies $g_j(M_\lambda,r) \leq g_1(N_\lambda,r)$, but we claim that in fact equality must hold. To wit, assume to the contrary that the inequality is strict; then the restriction of $M$ to the disc $|t| < e^{-g_j(M_\lambda,r)}$ would have a submodule isomorphic to $N_\lambda^\dual$.
This would mean that $M$ has a local horizontal section whose exact radius of convergence is equal to $e^{-g_1(N_\lambda,r)}$, which would imply that there exists some value of $k$ for which $g_k(M,r) = g_1(N_\lambda,r)$. However, this contradicts our choice of the interval $I$.

To summarize, for $r$ in a neighborhood of $r_0$ we have 
\[
g_j(M_\lambda,r) = \begin{cases} g_j(M, r) & (j=1,\dots,i) \\
g_j(N_\lambda,r) & (j=i+1,\dots,n).
\end{cases}
\]
We thus deduce the original claim by applying the case $i=n$ to $M_\lambda$.
\end{proof}

We now deduce Theorem~\ref{T:monotonicity}.

\begin{theorem} \label{T:monotonicity appendix2}
For $x,y \in \DD_\beta$ such that $x$ is the generic point of a disc containing $y$, for $i=1,\dots,n$, we have
\[
s_1(M,x) \cdots s_i(M,x) \leq s_1(M,y) \cdots s_i(M,y).
\]
In particular, the conclusion of Theorem~\ref{T:monotonicity} holds.
\end{theorem}
\begin{proof}
This follows from Lemma~\ref{L:monotonic2a} thanks to the invariance of the $s_i$
under base extension.
\end{proof}

We next proceed towards Theorem~\ref{T:virtual local index}(c).
\begin{defn}
Define the convergence polygon $\calN$ of $M$ as in Definition~\ref{D:convergence polygon general}, using the same disc at every point. Define the modified convergence polygon $\calN'$ using maximal discs not containing 0.
\end{defn}

\begin{lemma} \label{L:convexity}
Assume $p=0$. For $i=1,\dots,n$, for $x \in \DD_\beta$, 
we have $\Delta h_i(\calN)_x \geq 0$.
\end{lemma}
\begin{proof}
By base extension, we may reduce to the case $x = x_1$;
we may also assume that the norm on $K$ is nontrivial.
As in the proof of Lemma~\ref{L:monotonic2a}, we may further reduce to the case $i=n$.
We may further reduce to the case where $H^0(M_{[0,\delta)}) = 0$ for all $\delta \in (1,\beta)$.

Choose a nonempty set $W$ of $K$-rational points of $\DD_\beta$ with the following properties.
\begin{enumerate}
\item[(a)]
For each $w \in W$, $U_{w,1}$ is a branch of $\PP^1_K$ at $x_1$, which we also denote by
$\vec{t}_w$.
\item[(b)]
The branches $\vec{t}_w$ for $w \in W$ are pairwise distinct.
\item[(c)]
Let $\vec{t}_\infty$ be the branch of $\PP^1_K$ at $x$ in the direction of $\infty$.
Then for all branches $\vec{t}$ of $\PP^1_K$ at $x_1$, 
we have $\partial_{\vec{t}}(\calN) = 0$
unless $\vec{t} = \vec{t}_\infty$ or $\vec{t} = \vec{t}_w$ for some $w \in W$.
\end{enumerate}
Let $R_{w,I}$ be the ring of rigid analytic functions on the space $|z-w| \in I$,
and put $M_{w,I}= M \otimes_{R_{[0,\beta)}} R_{w,I}$. Then
by Lemma~\ref{L:H1 slope}, to prove the desired result, it suffices to check that
for any $\gamma \in (0,1)$, $\delta \in (1,\beta)$
sufficiently close to 1,
\[
\dim_K H^1(M_{[0,\delta)}) \geq \sum_{w \in W} \dim_K H^1(M_{w,[0,\gamma)}).
\]
It would hence also suffice to prove surjectivity of the map
\begin{equation} \label{eq:h1 surjective}
H^1(M_{[0,\delta)}) \to \bigoplus_{w \in W} H^1(M_{w,[0,\gamma)}).
\end{equation}
Since $p=0$, we may invoke Theorem~\ref{T:local index formula} to see that $H^1(M_{w,[0,\gamma)})$ is a finite-dimensional $K$-vector space, and so is complete with respect to its natural topology as a $K$-vector space (i.e., the one induced by the supremum norm with respect to some basis). Since the map $M_{w,[0,\gamma)} \to H^1(M_{w,[0,\gamma)})$ is a $K$-linear surjection from a Fr\'echet space over $K$ to a Banach space over $K$, the Banach open mapping theorem implies that this map is a quotient map of topological spaces. In particular, since the map
\[
M_{[0, \delta)} \to \bigoplus_{w \in W} M_{w,[0,\gamma)}
\]
has dense image, so then does
\eqref{eq:h1 surjective}, proving the claim.
\end{proof}

\begin{theorem} \label{T:convexity}
The conclusion of Theorem~\ref{T:virtual local index}(c) holds.
\end{theorem}
\begin{proof}
By Theorem~\ref{T:virtual local index}(b), we may assume $x \notin \Gamma_{X,Z}$;
the claim thus reduces to Lemma~\ref{L:convexity}.
\end{proof}

\begin{remark}
In the proof of Lemma~\ref{L:convexity}, to deduce the finite-dimensionality of 
$H^1(M_{w,[0,\gamma)})$, one may replace the invocation of Theorem~\ref{T:local index formula} with the following argument. For $\delta_1 \in (0, \gamma)$, $\delta_2 \in (\delta_1, \gamma)$ sufficiently large, by \cite[Lemma~3.7.6]{kedlaya-radii} we have
\[
\dim_K H^0(M_{w,(\delta_1, \gamma)}) =
\dim_K H^1(M_{w,(\delta_1, \gamma)}), \quad
\dim_K H^0(M_{w,(\delta_1, \delta_2)}) =
\dim_K H^1(M_{w,(\delta_1, \delta_2)}).
\]
By this calculation plus Mayer-Vietoris, the map $H^1(M_{w,[0,\gamma)}) \to H^1(M_{w,[0,\delta_2)})$
has finite-dimensional kernel and cokernel. 
Since $R_{w,[0,\gamma)} \to R_{w,[0,\delta_2)}$ is a compact map of topological $K$-vector spaces, the same is true of $M_{w,[0,\gamma)} \to M_{w,[0,\delta_2)}$; we may thus apply the Schwartz-Cartan-Serre lemma \cite[Satz~1.2]{kiehl} to conclude.
\end{remark}

\begin{remark} \label{R:convexity positive characteristic}
At this point, it is natural to consider what happens when $p>0$.
If we also add suitable hypotheses on $p$-adic non-Liouville exponents, then all of the preceding statements remain true
(as in Remark~\ref{R:finite morphism liouville}).
Absent such hypotheses, we cannot rely on finite-dimensionality of cohomology groups, but it may nonetheless be possible to adapt the proof of Theorem~\ref{T:convexity} to the case $p>0$
by establishing a relative version of Lemma~\ref{L:H1 slope}. To be precise, in the notation of the proof of Lemma~\ref{L:convexity}, one might hope to prove that
\[
H^1(M_{[0,\delta)} / \bigoplus_{w \in W} M_{w,[0,\gamma)}) = 0
\]
and to relate this vanishing directly to the Laplacian, bypassing the potential failure of finite-dimensionality for the individual cohomology groups.
\end{remark}

\section{Thematic bibliography}
\label{sec:thematic bib}

As promised in the introduction, we include an expansive but unannotated list of references related to key topics in the paper. For those topics discussed in \cite{kedlaya-book}, the chapter endnotes therein may be consulted for additional context.

\begin{itemize}

\item
Underlying structure of Berkovich spaces: 
\cite{bpr2}, \cite{bpr}, 
\cite[Chapter~1]{baker-rumely},
\cite{berkovich-contractible},
\cite{berkovich-contractible2},
\cite{cd},
\cite{ducros1}, \cite{ducros2}, \cite{ducros},
\cite{grw},
\cite{hrushovski-loeser},
\cite{payne},
\cite{poineau-angelique},
\cite{tyomkin}.
See also the survey \cite{werner-pr} in this volume.

\item
Potential theory for Berkovich curves: \cite[chapter by Baker]{aws}, \cite{baker-rumely}, \cite{thuillier}.

\item
Formal structure of singular connections:  
\cite[Chapter~4]{dgs},
\cite{gerard-levelt},
\cite[\S 11]{katz-turrittin},
\cite[Chapter~7]{kedlaya-book},
\cite{kedlaya-goodformal1}, \cite{kedlaya-goodformal2},
\cite{levelt}, 
\cite{levelt-vandenessen}, 
\cite{malgrange}, \cite{malgrange2},
\cite{mochizuki1}, \cite{mochizuki2}, 
\cite[Chapter~3]{svdp},
\cite{turrittin}.

\item
Convergence of solutions of $p$-adic differential equations:
\cite{baldassarri-singular},
\cite{baldassarri},
\cite{baldassarri-divizio},
\cite{baldassarri-kedlaya},
\cite{christol-book},
\cite{christol},
\cite{christol-dwork-couronnes},
\cite{christol-remmal},
\cite{clark}, 
\cite{dwork-robba-natural},
\cite[Chapters~9--11]{kedlaya-book},
\cite{kedlaya-xiao},
\cite{lutz},
\cite{pulita-poineau2}, \cite{pulita-poineau3}, 
\cite{pons1}, \cite{pons}, \cite{pulita-poineau}, 
\cite{young}.
For a retrospective circa 2000, see also  \cite{christol-retro}.

\item
Transfer principles and effective convergence bounds:
\cite{christol-transfer},
\cite{christol-dwork-effective},
\cite{dwork-robba},
\cite{dwork-robba-effective},
\cite[Chapters~9, 13, 18]{kedlaya-book}.

\item
Logarithmic growth of $p$-adic solutions:
\cite{andre-dwork-conj},
\cite{chiarellotto-tsuzuki1},
\cite{chiarellotto-tsuzuki2},
\cite{christol-book},
\cite{dwork-pde2}, \cite{dwork-pde3},
\cite[Chapter~18]{kedlaya-book},
\cite{manjra},
\cite{ohkubo1}, \cite{ohkubo2}.

\item
Stokes phenomena for complex connections:
\cite{loday-richaud}, \cite{loday-richaud2}, \cite{majima}, 
\cite{mochizuki3}, \cite{sabbah}, \cite{sabbah2}, \cite[Chapters 7--9]{svdp}, \cite{varadarajan}.

\item
Index formulas for nonarchimedean differential equations:
\cite{christol-book2},
\cite{cm1}, \cite{cm2}, \cite{cm3}, \cite{cm4}, 
\cite{kedlaya-overview}, 
\cite{kedlaya-weil2},
\cite{matsuda},
\cite{robba-index1}, \cite{robba-index2}, \cite{robba-index3}, \cite{robba-indice4},
\cite{pulita-poineau5}, \cite{tsuzuki-index}, \cite{young}.
See also the thematic bibliography of \cite{robba-cohom},
and \cite{cm-fez} for a survey of \cite{cm1,cm2,cm3,cm4}.

\item
Comparison between algebraic and complex analytic cohomology of connections:
\cite{andre-baldassarri}, 
\cite{atiyah-hodge}, \cite{deligne-eq}, \cite{grothendieck-derham}, 
\cite{mebkhout-comparison}.

\item
Comparison between algebraic and nonarchimedean analytic cohomology of connections:
\cite{andre-comparison},
\cite{andre-baldassarri},
\cite{baldassarri-comparison}, \cite{baldassarri-comparison2}, 
\cite{chiarellotto-comparison}, \cite{chiarellotto-comparison2}, \cite[chapter by Kedlaya]{aws}, \cite{kiehl-comparison}.

\item
Decomposition theorems for nonarchimedean connections:
\cite{christol-book}, 
\cite{christol-book2},
\cite{cm1}, \cite{cm2}, \cite{cm3}, \cite{cm4}, \cite{dwork-robba}, \cite[Chapter~12]{kedlaya-book},
\cite{kedlaya-radii}, \cite{pulita-poineau4}, \cite{robba-index1}, \cite{robba-hensel}.

\item
Monodromy for $p$-adic connections (Crew's conjecture):
\cite{andre-monodromy}, 
\cite{crew},
\cite{kedlaya-monodromy},
\cite{kedlaya-semistable1},
\cite{kedlaya-semistable2},
\cite{kedlaya-semistable3},
\cite[Chapters~20, 21]{kedlaya-book},
\cite{kedlaya-semistable4},
\cite{kedlaya-radii},
\cite{mebkhout-monodromy},
\cite{tsuzuki-monodromy},
\cite{tsuzuki-monodromy2}.

\item
$p$-adic Liouville numbers and $p$-adic exponents:
\cite{christol-book2},
\cite{cm1}, \cite{cm2}, \cite{cm3}, \cite{cm4}, \cite{clark},
\cite{dwork-exponents}, 
\cite[Chapter~13]{kedlaya-book}, \cite{kedlaya-radii}.

\item
Ramification of maps of Berkovich curves:
\cite{temkin1}, \cite{faber1}, \cite{faber2}, \cite{huber}, 
\cite[Chapter~19]{kedlaya-book},  \cite{temkin2}.

\item
Measures of ramification and $p$-adic differential equations:
\cite{baldassarri-rolle}, \cite{chiarellotto-pulita},
\cite{kedlaya-overview}, 
\cite{kedlaya-swan1}, \cite[Chapter~19]{kedlaya-book}, \cite{kedlaya-swan2}, 
\cite{marmora}, \cite{matsuda1}, \cite{matsuda}, \cite{matsuda-conj}, \cite{ohkubo3}, \cite{tsuzuki-index}, \cite{xiao-ramif1}, \cite{xiao-ramif2}, \cite{xiao-refined}.

\item
Kummer theory in mixed characteristic (Kummer-Artin-Schreier-Witt theory):
\cite{green-matignon},
\cite{matsuda1},
\cite{mezard-romagny-tossici},
\cite{oss}, \cite{pulita-rank1},
\cite{sekiguchi-suwa}, \cite{sekiguchi-suwa1} (unpublished; see
\cite{mezard-romagny-tossici} instead), \cite{sekiguchi-suwa2}, \cite{tossici},
\cite{waterhouse}.

\item
Oort lifting problem: 
\cite{bertin},
\cite{bouw-wewers},
\cite{brewis-wewers}, 
\cite{cgh1}, \cite{cgh},
\cite{corry},
\cite{garuti},
\cite{green-matignon},
\cite{green-matignon2},
\cite[\S 9]{hops},
\cite{obus12},
\cite{obus-gen},
\cite{obus-wewers},
\cite{oort},
\cite{oss},
\cite{pagot},
\cite{pop}, \cite{saidi}.
See also the lecture notes \cite{bouw-wewers-aws}, the introductions to \cite{obus-wewers, pop}, and the PhD thesis \cite{turchetti}.

\end{itemize}

\end{document}